\pdfoutput=1
\documentclass[11pt,a4paper]{article}
\usepackage{amsthm,amsmath,amssymb}
\usepackage{graphicx,color}
\usepackage{mathrsfs}

\newtheorem{thm}{Theorem}[section]
\newtheorem{Lem}{Lemma}[section]

\newtheorem{conjecture}{Conjecture}[section]
\newtheorem{defn}{Definition}[section]
\newtheorem{exam}{Example}[section]
\newtheorem{rem}{Remark}[section]

\usepackage[varg]{txfonts}%%!!
\makeatletter%
\input{ot1txtt.fd}
\makeatother%
\makeatother

\begin{document}
	\begin{center}
		\LARGE {\bf Impartial Triangular Chocolate Bar Games}
	\end{center}
	
	\begin{center}
		\large Ryohei Miyadera\footnote{Kwansei Gakuin High School}, Shunsuke Nakamura\footnote{Osaka University}, Masanori Fukui\footnote{Hyogo University of Teacher Education}
	\end{center}

\begin{abstract}
	Chocolate bar games are variants of the game of Nim in which the goal is to leave your opponent with the single bitter part of the chocolate bar. The rectangular chocolate bar game is a thinly disguised form of classical multi-heap Nim. In this work, we investigate the mathematical structure of triangular chocolate bar games in which the triangular chocolate bar can be cut in three directions. 
	In the triangular chocolate bar game, a position is a $\mathcal{P}$-position if and only if $x \oplus y \oplus z = 0$, where the numbers $x,y,z$
	 stand for the maximum number of times that the chocolate bar can be cut in each direction. 
	Moreover, the Grundy number of a position $(x,y,z)$ is not always equal to $x \oplus y \oplus z $, and a generic formula for Grundy numbers in not known. Therefore, the mathematical structure of triangular chocolate bar game is different from that of classical Nim. 
\end{abstract}

\section{Introduction}\label{intro}
The original chocolate bar game \cite{robin} consists of square boxes in which one square is blue and other squares are brown.
Brown squares are sweet, and the blue square is considered too bitter to eat. For example, see Figure \ref{robinchoco1}.
Each player takes turns breaking the bar in a straight line along the grooves and eating the piece that does not contain the bitter part. The player who breaks the chocolate bar and leaves his opponent with the single bitter blue square is the winner.
Since the horizontal and vertical grooves are independent, the rectangular chocolate bar of Figure \ref{robinchoco1} is equivalent to a game of Nim with three heaps of 3 stones, 7 stones and 4 stones. Therefore, a rectangular chooclate bar game is mathematically the same as the game of Nim \cite{bouton}.
	
We have previously studied chocolate bar games such as those shown in Figure \ref{demochocolate4} and Figure \ref{2yzchoco1} in \cite{integer2015}.
In these chocolate bar games, a vertical break can
reduce the number of horizontal breaks, and the mathematical structure of these games is different from that of classical Nim and rectangular chocolate bar games. We can still consider the game as being
played with heaps, but now a single move may change more than one heap.
We have also studied chocolate bars games such as those shown in Figure \ref{demochocolate4} and Figure \ref{2yzchoco1} with a pass move in \cite{integer2016}.

There are other types of chocolate bar games. One of the most well known chocolate bar games is CHOMP \cite{gale}, which uses a rectangular chocolate bar. The players take turns to
choose one block, and the players eat this block together with the blocks below it and to its right. The top left block is bitter, and the players cannot eat this block. Although many people have studied this game, the winning strategy has yet to be discovered. 

In this study, we consider triangular chocolate bar games such as that shown in Figure \ref{demochocolate3}.
A triangular chocolate bar can be cut in three directions. We previously studied a simple type of triangular chocolate bar game in \cite{ipsj1}, and 
the results of \cite{ipsj1} are generalized in this article.

In Section \ref{sectionforcomputer}, Mathematica programs and CGSuite programs are presented. By these programs, readers can check the results of this article with their own computers.

	\begin{exam}
		Examples of chocolate bar games.

\begin{figure}[!htb]
	\begin{minipage}[!htb]{0.45\columnwidth}
		\centering
		\includegraphics[width=0.5\columnwidth,bb=0 0 260 292]{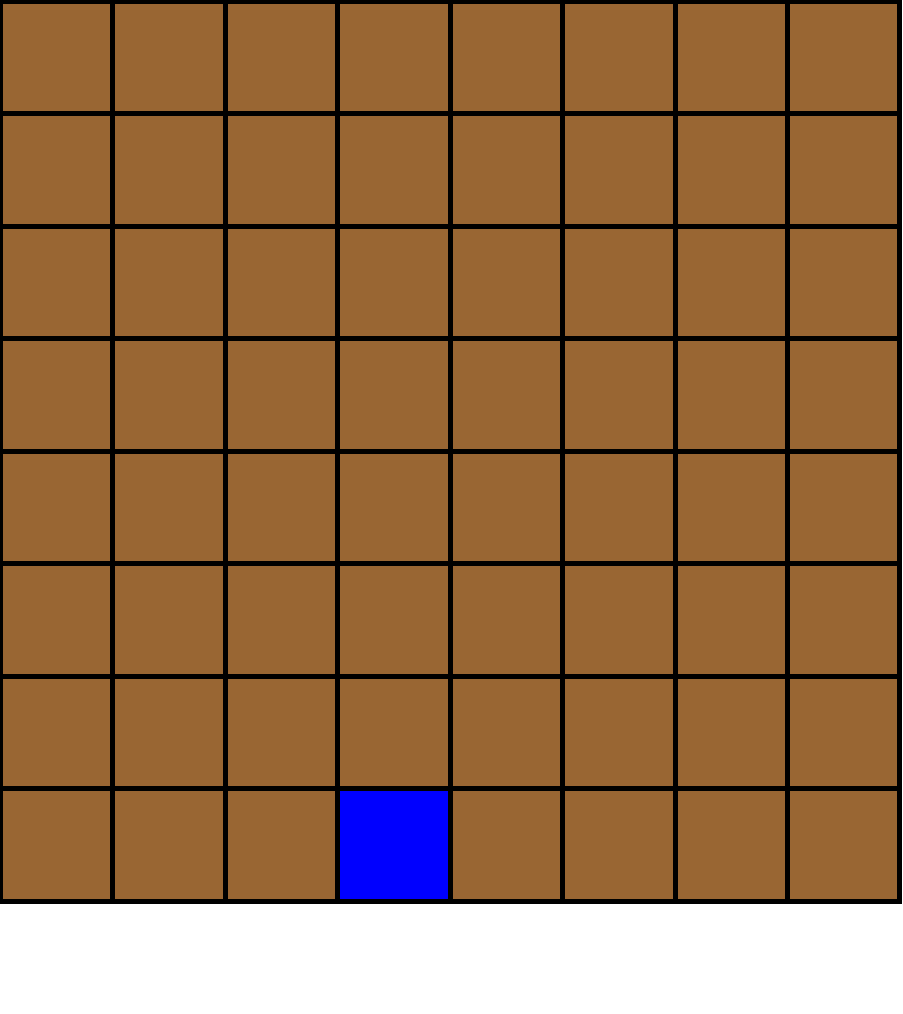}
		\caption{example (1)}
		\label{robinchoco1}
	\end{minipage}
	\begin{minipage}[!htb]{0.45\columnwidth}
		\centering
		\includegraphics[width=0.8\columnwidth,bb=0 0 260 58]{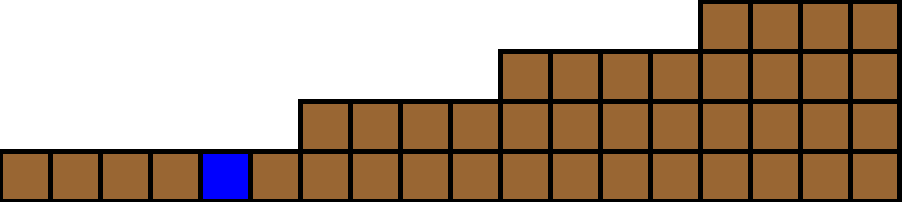}
		\caption{example (2)}
		\label{demochocolate4}
	\end{minipage}
\end{figure}
	
\begin{figure}[!htb]
	\begin{minipage}[!htb]{0.45\columnwidth}
		\centering
		\includegraphics[width=0.7\columnwidth,bb=0 0 260 93]{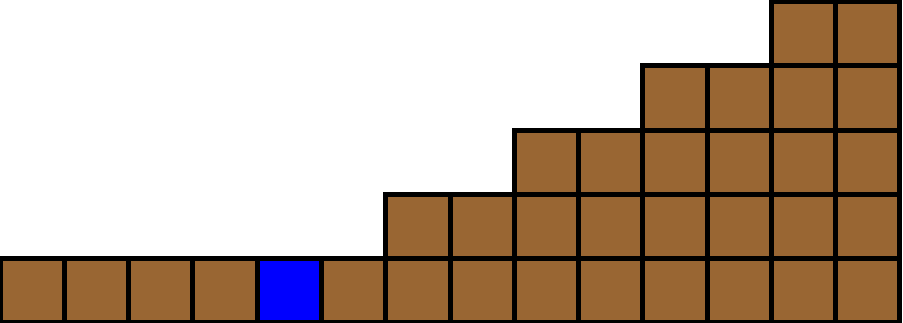}
		\caption{example (3)}
		\label{2yzchoco1}
	\end{minipage}
	\begin{minipage}[!htb]{0.45\columnwidth}
		\centering
		\includegraphics[width=0.6\columnwidth,bb=0 0 116 57]{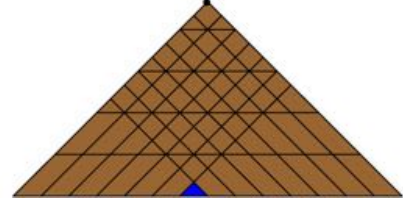}
		\caption{example (4)}
		\label{demochocolate3}
	\end{minipage}
\end{figure}							
\end{exam}

\section{Definitions and Theorems of Game Theory}
Throughout this study, we denote the set of non-negative integers by $Z_{\geq0}$, and $N$ is the set of natural numbers.
For completeness, we quickly review the necessary game theory concepts used in this study; see \cite{lesson} or \cite{combysiegel} for more details. 

As chocolate bar games are impartial games without draws, there will only be two outcome classes.
\begin{defn}\label{NPpositions}
 	$(i)$ $\mathcal{N}$-positions are positions from which the next player can force a win, as long as he plays correctly at every stage.\\
 	$(ii)$ $\mathcal{P}$-positions are positions from which the previous player (the player who will play after the next player) can force a win, as long as he plays correctly at every stage.
\end{defn}

The outcome of this game is not pre-determined; however, there is nothing that the potential loser can do if the potential winner plays correctly at ever stage. The potential winner cannot afford to make a single mistake, or his opponent can exploit the mistake and win the game.

One of the most important aims in the study of chocolate bar games is the identification of all $\mathcal{P}$-positions and $\mathcal{N}$-positions. 

\begin{defn}\label{sumofgames}
 The \textit{disjunctive sum} of two games, denoted by $G+H$, is a super-game in which a player may move either in $G$ or $H$, but not in both.
\end{defn}

\begin{defn}
For any position $\mathbf{p}$, there exists a set of positions that can be reached by making precisely one move from $\mathbf{p}$, which we will denote by \textit{move}$(\mathbf{p})$. 
\end{defn}

Examples \ref{chocoexmp1} and \ref{defofmovek} demonstrate the use of \textit{move}.

\begin{defn}\label{defofmexgrundy}
	$(i)$ The \textit{minimum excluded value} $(\textit{mex})$ of a set, $S$, of non-negative integers is the smallest non-negative integer not in S. \\
	$(ii)$ Each position $\mathbf{p}$ of an impartial game has an associated Grundy number, which is denoted by $\mathcal{G}(\mathbf{p})$.
	The Grundy number of the end position is $0$, and the Grundy number is found recursively for all other positions: 
	$\mathcal{G}(\mathbf{p}) = \textit{mex}\{\mathcal{G}(\mathbf{h}): \mathbf{h} \in move(\mathbf{p})\}.$
\end{defn}

The power of the Sprague--Grundy theory for impartial games is contained in the following theorem.

\begin{thm}\label{thmforsumofgame}
	Let $G$ and $H$ be impartial games, and let $\mathcal{G}_{G}$ and $\mathcal{G}_{H}$ be the Grundy numbers of $G$ and $H$, respectively. Then, the following relationships hold:\\
	$(i)$ For any position $\mathbf{g}$ of $G$ we have
	$\mathcal{G}_{G}(\mathbf{g})=0$ if and only if $\mathbf{g}$ is a $\mathcal{P}$-position.\\
	$(ii)$ The Grundy number of a position $\{\mathbf{g},\mathbf{h}\}$ in the game $G+H$ is
	$\mathcal{G}_{G}(\mathbf{g})\oplus \mathcal{G}_{H}(\mathbf{h})$.
\end{thm}

Please see \cite{lesson} for a proof of this theorem.

Finally, we define nim-sum, which is important for the theory of chocolate bar games.

\begin{defn}\label{definitionfonimsum11}
	Let $x,y$ be non-negative integers written in base $2$ so that $x = \sum\limits_{i = 0}^n {{x_i}} {2^i}$ and $y = \sum\limits_{i = 0}^n {{y_i}} {2^i}$ with ${x_i},{y_i} \in \{0,1\}$.\\
	We define the nim-sum $x \oplus y$ by
	\begin{align}
		x \oplus y = \sum\limits_{i = 0}^n {{w_i}} {2^i},
	\end{align}
	where $w_{i}=x_{i}+y_{i} \ (mod\ 2)$.
\end{defn}

When we use $ \sum\limits_{i = 0}^n {{x_i}} {2^i}$ and $y = \sum\limits_{i = 0}^n {{y_i}} {2^i}$, we assume that at least one $x_n$ and $y_n$ term is not zero.

\section{Rectangular Chocolate Bar Games}\label{rectangle}
We first define rectangular chocolate bar games. Please consult the chocolate bar in Figure \ref{robinchoco} as examples for definitions \ref{defofgeneralchoco} and \ref{defofchocoandcoordinate}. 

\begin{defn}\label{defofgeneralchoco}
The chocolate bar consists of square boxes, where one block is blue and the others are brown.
Brown blocks are sweet, and the blue block is considered too bitter to eat. This game is played by two players in turn. 
Each player breaks the chocolate (along a black line) into two areas.
The player eats the area that does not contain the bitter blue block. The player who breaks the chocolate and leaves his opponent with the single bitter blue block is the winner.
\end{defn}

\begin{exam}
The chocolate bars shown in Figure \ref{robinchoco} and Figure \ref{robinchocog} were proposed by Robin \cite{robin}.

\begin{figure}[!htb]
	\begin{minipage}[!htb]{0.45\columnwidth}
		\centering
		\includegraphics[width=0.6\columnwidth,bb=0 0 260 292]{robinchoco3.pdf}
		\caption{Position (3,7,4)}
		\label{robinchoco}
	\end{minipage}
	\begin{minipage}[!htb]{0.45\columnwidth}
		\centering
		\includegraphics[width=0.6\columnwidth,bb=0 0 170 153]{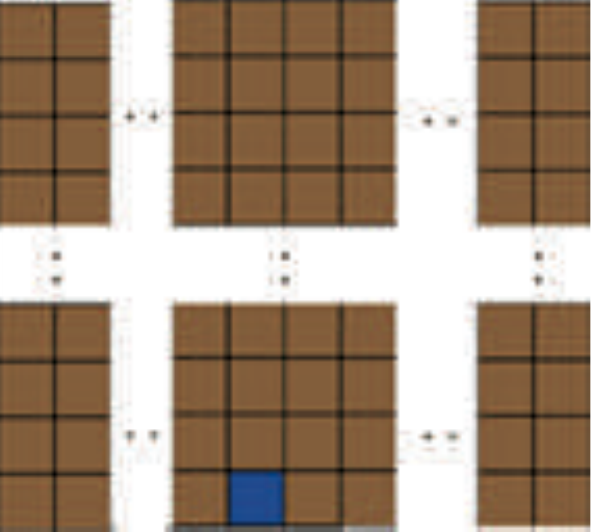}
		\caption{Position (x,y,z)}
		\label{robinchocog}
	\end{minipage}	
\end{figure}

\end{exam}
	
\begin{defn}\label{defofchocoandcoordinate}
	We can cut these chocolates along the segments in three ways:\\
	$(i)$ vertically on the left side of the bitter blue block;\\
	$(ii)$ horizontally above the bitter blue block; and\\
	$(iii)$ vertically on the right side of the bitter blue block.\\
	Therefore, this chocolate bar can be represented with $(x,y,z)$, where $x,y,z$ stand for the maximum number of times the chocolate bar can be cut in each direction. 
\end{defn}
	
\begin{rem}
	For example, in Figure\ref{robinchoco}, we can make at most three vertical cuts on the left side of the bitter blue block, seven horizontal cuts above the bitter blue block, and four vertical cuts on the right side of the bitter blue block. Therefore, we have $x=3$, $y = 7$, and $z=4$, and we represent the chocolate bar in Figure \ref{robinchoco} with the position $(3,7,4)$.
	We can make chocolates as large as we want. For example, the chocolate bar in Figure \ref{robinchocog} has position $(x,y,z)$, where $x,y,z \in Z_{\geq0} $.
	In Definition \ref{defofchocoandcoordinate}, we introduce the three ways to cut rectangular chocolates.
	Example \ref{chocoexmp1} shows how to cut chocolates, and we define how to cut chocolate mathematically in Definition \ref{defofmoverect}. 
\end{rem}

\begin{rem}
	Note that, in Definition \ref{defofgeneralchoco}, we are not considering a mis\`ere play since the player who breaks the chocolate bar for the last time is the winner. Therefore, the chocolate bar games in this paper are normal play games.
\end{rem}

\begin{defn}\label{defofmoverect}
	For $x,y,z \in Z_{\ge 0}$, we define $move((x,y,z))=\{(u,y,z);0 \leq u<x) \cup \{(x,v,z);0 \leq v<y) \cup \{(x,y,w);0 \leq w<z\}$, where $u,v,w \in Z_{\ge 0}$.
\end{defn}

	Here, $move((x,y,z))$ is the set of states that can be directly reached from the state $(x,y,z)$.
	
\begin{exam}\label{chocoexmp1}
	$(i)$ If we start with the chocolate bar in Figure \ref{robinchoco} and cut vertically to remove three columns to the right of the blue part, we get the chocolate bar in Figure \ref{robinchoco31}.
	Then, from the chocolate bar in Figure \ref{robinchoco}, we cut horizontally to remove the seven rows above the blue part, and we get the chocolate bar in Figure \ref{robinchocol1}.
	Therefore, $(3,7,1)$, $(3,0,4)$$\in move((3,7,4))$.
	
	\begin{figure}[!htb]
		\begin{minipage}[!htb]{0.45\columnwidth}
			\centering
	\includegraphics[width=0.6\columnwidth,bb=0 0 176 279]{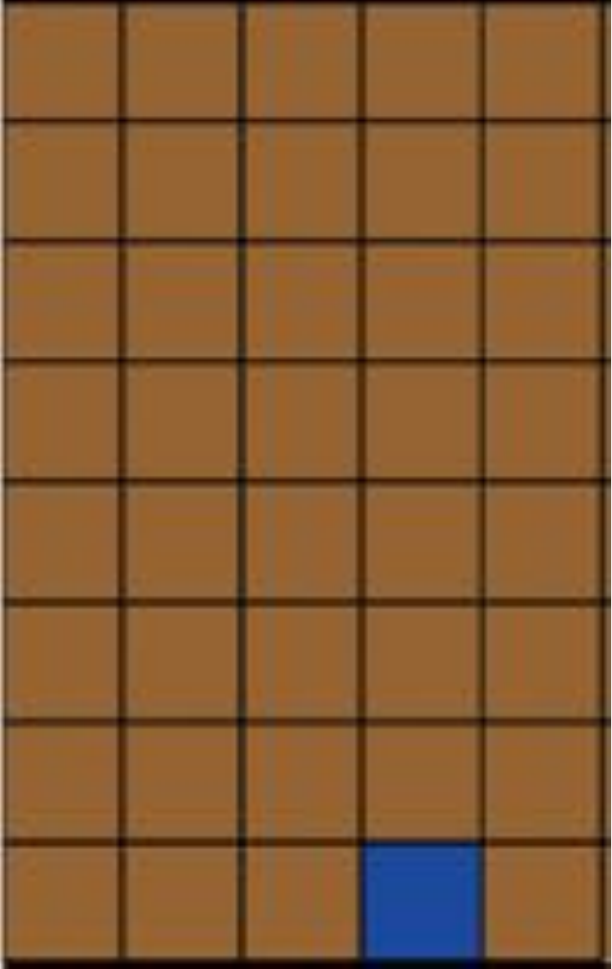}
	\caption{Position (3,7,1)}
		\label{robinchoco31}
		\end{minipage}
		\begin{minipage}[!htb]{0.45\columnwidth}
			\centering
	\includegraphics[width=0.6\columnwidth,bb=0 0 57 43]{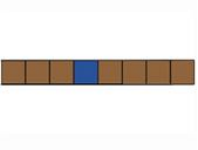}
		\caption{Position (3,0,4)}
		\label{robinchocol1}
		\end{minipage}
	\end{figure}
		
\noindent
$(ii)$ If we start with the chocolate bar in Figure \ref{robinchoco}, we cannot directly go to the chocolates in Figures \ref{robinchoco32}, \ref{robinchoco3l2}, and \ref{robinchoco33}. Therefore, $(3,2,1)$, $(0,4,4)$, $(0,0,0) \notin move((3,7,4))$.
	
\noindent
On the other hand, we can move directly from Figure \ref{robinchoco31} to Figure \ref{robinchoco32} by removing five rows at the same time. Therefore, $(3,2,1) \in move((3,7,1))$.

\begin{figure}[!htb]
\centering
\includegraphics[width=0.05\columnwidth,bb=0 0 35 176]{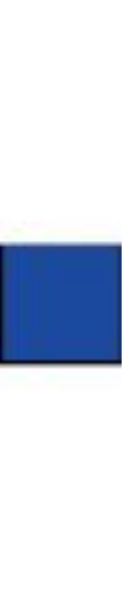}
\vspace{-10mm}
\caption{Position (0,0,0)}
\label{robinchoco33}
\end{figure}

\begin{figure}[!htb]
	\begin{minipage}[!htb]{0.45\columnwidth}
		\centering
		\includegraphics[width=0.6\columnwidth,bb=0 0 175 107]{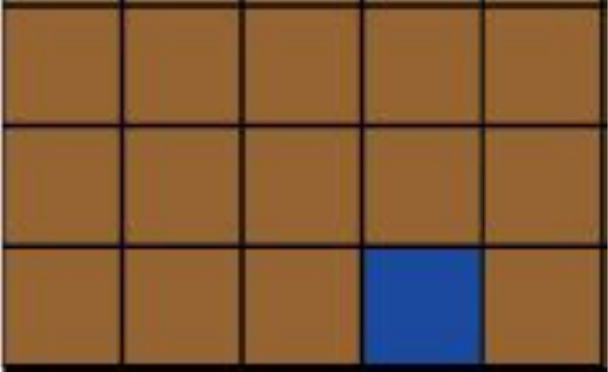}
		\caption{Position (3,2,1)}
		\label{robinchoco32}
	\end{minipage}
	\begin{minipage}[!htb]{0.45\columnwidth}
		\centering
		\includegraphics[width=0.4\columnwidth,bb=0 0 113 115]{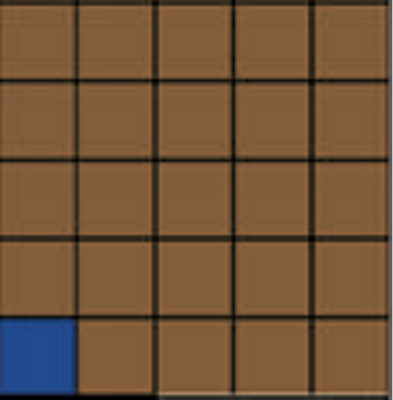}
		\caption{Position (0,4,4)}
		\label{robinchoco3l2}
	\end{minipage}
\end{figure}
\end{exam}
	
\begin{rem}
	It is easy to prove that the chocolate bar in Figure \ref{robinchoco3l2} is a $\mathcal{P}$-position. Note that the numbers of rows and columns are the same in the initial chocolate bar. With the first move, the number of columns will be different from the number of the rows. Then, the opposing player can break the bar to make the numbers of rows and columns the same. In this way, the opposing player will always keep the numbers of rows and columns the same, winning the game by moving to the single bitter block of Figure \ref{robinchoco33} that is represented by the position $(0,0,0)$ . 
\end{rem}

\begin{thm}
A position $(x,y,z)$ of a rectangular chocolate bar is a $\mathcal{P}$-position if and only if $x \oplus y \oplus z = 0$.
\end{thm}	

For the proof of this theorem, please see \cite{robin}.

\section{Triangular Chocolate Bar Games}\label{introforkyxzgame}
Herein, we assume that $k$ is a fixed natural number such that $k = 4m+3$ for some $m \in Z_{\geq0}$.
In this section, we study triangular chocolate bar games.
Examples of triangular chocolate bars are presented in Example \ref{exampleoftrichoco}.

\begin{exam}\label{exampleoftrichoco} Triangular chocolate bars are shown in Figure \ref{choco747a} and Figure \ref{choco929a}.
These chocolate bars consist of polygons, where only one triangle is blue and other polygons are brown.
Brown polygons are sweet, and the blue triangle is considered too bitter to eat. 

A triangular chocolate bar can be cut in three ways: $(i)$ diagonally from the upper right to the lower left above the blue triangle; $(ii)$ horizontally above the blue triangle; or $(iii)$ diagonally from the upper left to the lower right above the blue triangle.
In Figure \ref{choco747a}, the numbers $7,4,7$ represent the maximum number of times we can cut the chocolate bar directions $(i)$, $(ii)$, and $(iii)$, respectively; in 
Figure \ref{choco929a}, the numbers $9,2,9$ represent the maximum number of times we can cut the chocolate bar directions $(i)$, $(ii)$, and $(iii)$, respectively.	

\begin{figure}[!htb]
	\begin{minipage}[!htb]{0.45\columnwidth}
		\centering
		\includegraphics[width=0.7\columnwidth,bb=0 0 222 113]{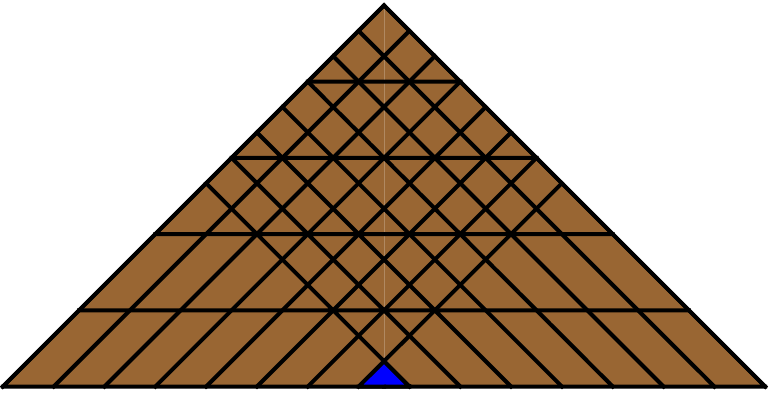}
		\caption{Position $(7,4,7)$}
		\label{choco747a}
	\end{minipage}
\begin{minipage}[!htb]{0.05\columnwidth}
	~
\end{minipage}	
	\begin{minipage}[!htb]{0.45\columnwidth}	
		\centering
		\includegraphics[width=0.7\columnwidth,bb=0 0 228 115]{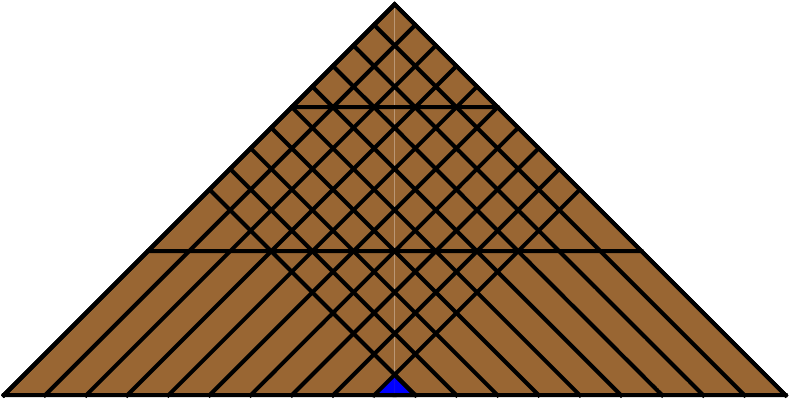}
		\caption{Position $(9,2,9)$}
		\label{choco929a}
		\end{minipage}
\end{figure}
	
\end{exam} 

We more precisely define the coordinates and cuts available in triangular chocolate bars in Definition \ref{definitionoflines} and 
Definition \ref{definitionofchoco}.

\begin{defn}\label{definitionoflines}
In the $x$ and $y$ coordinate system, we draw three groups of lines:
$(i)$ $y = x+1+2r$ for $r \in Z_{\geq0}$.\\
$(ii)$ $y = ks$ for $s \in N$.\\
$(iii)$ $y = -x+1+2t$ for $t \in Z_{\geq0}$.
\end{defn} 

The lines defined in Definition \ref{definitionoflines} are shown in Figure \ref{keq3lines} and 
Figure \ref{keq5lines} for $k=3$ and $k=7$, respectively.

\begin{figure}[!htb]
\begin{minipage}[!htb]{0.45\columnwidth}
	\centering
\includegraphics[width=0.7\columnwidth,bb=0 0 260 137]{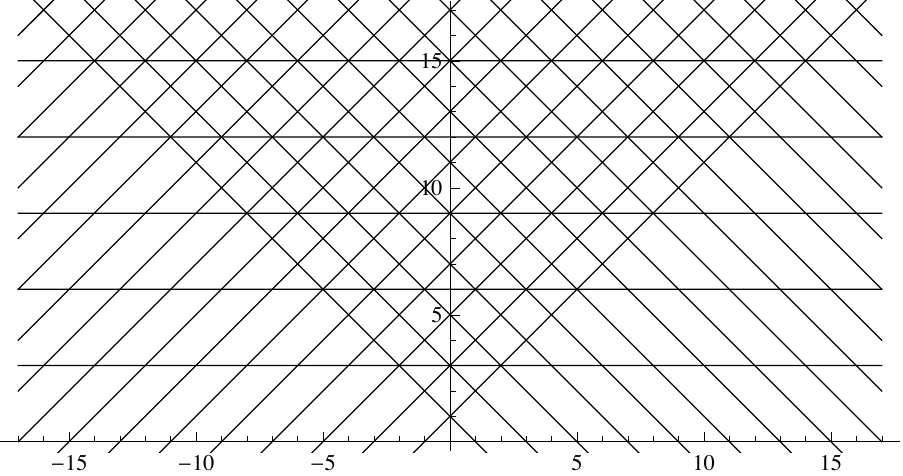}
\caption{lines for $k=3$}
\label{keq3lines}
\end{minipage}
\begin{minipage}[!htb]{0.05\columnwidth}
	~
\end{minipage}	
\begin{minipage}[!htb]{0.45\columnwidth}
	\centering
\includegraphics[width=0.7\columnwidth,bb=0 0 260 160]{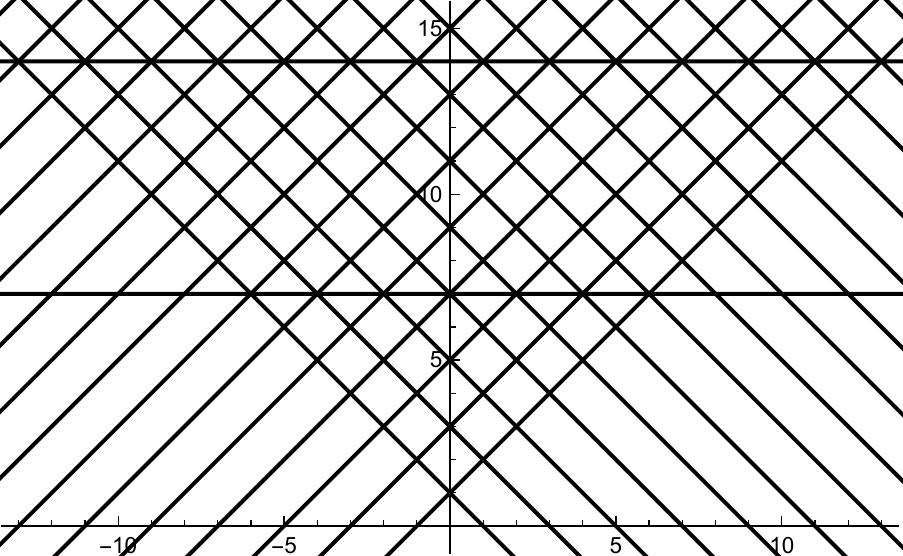}
\caption{lines for $k=7$}
\label{keq5lines}
\end{minipage}
\end{figure}

\begin{defn}\label{definitionofchoco}
Let $u,v,w \in Z_{\geq0}$ such that $kv \leq u+w$.\\
We denote the area of the chocolate bar described by the following four inequalitiesas position $(u,v,w)$:\\
$(a)$ $y \leq x+1+2u$;\\
$(b)$ $y \leq k(v+1)$;\\
$(c)$ $y \leq -x+1+2w$; and\\
$(d)$ $y \geq 0$.\\
The areas denoted by $(u,v,w)$ are colored brown, for except the triangular area defined by 
$y \leq x+1, y \leq -x+1$ and $ y \geq 0$, which is colored blue.

This chocolate bar game is played by two players in turn. 
Each player breaks the chocolate bar (along a black line) into two areas. 
The player eats the area that does not contain the bitter blue triangle. The player who breaks the chocolate bar and leaves his opponent with the single bitter blue triangle is the winner.	
The three numbers $u,v,w$ represent the maximum number of times we can cut the chocolate bar each direction.
\end{defn}

\begin{exam}\label{examkequal3}
Let $k=3$. Here, we present examples of triangular chocolate bars of position $(u,v,w)$ for $u,v,w \in Z_{\ge 0}$
such that $kv \leq u + w$.

\noindent
$(i)$ Let $(u,v,w) =(7,4,7)$. Then, by inequalities in $(a)$, $(b)$, $(c)$, and $(d)$ of Definition \ref{definitionofchoco}, 
 we have the following inequalities:
 
\begin{figure}[!htb]
	\centering
\includegraphics[width=0.45\columnwidth,bb=0 0 260 144]{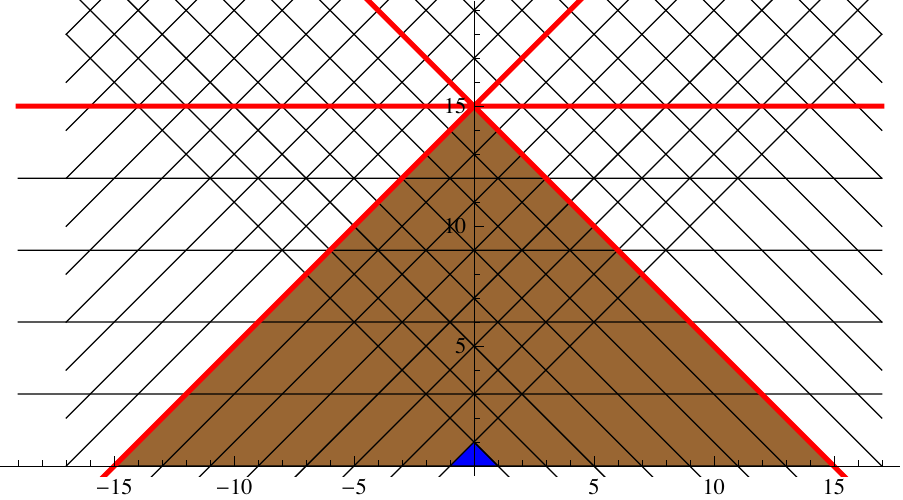}	
\caption{Position $(7,4,7)$}
\label{demok31}	
\end{figure}

\noindent
$(a.1)$ $y \leq x + 15$;\\
$(b.1)$ $y \leq 15$;\\
$(c.1)$ $y \leq -x+15$; and\\
$(d.1)$ $y \geq 0$.\\
Lines $y = x + 15$, $y = 15$, and $y = -x+15$ are represented by red lines in Figure \ref{demok31}. 
Here, we can omit inequality $(b.1)$ since the inequalities in $(a.1)$ and $(c.1)$ imply inequality $(b.1)$.\\	
It is easy to see that the three numbers $7,4,7$ represent the maximum number of times we can cut the chocolate bar each direction.

\noindent
$(ii)$ Let $(u,v,w) =(5,4,7)$. Then, by the inequalities in $(a)$, $(b)$, $(c)$, $(d)$ of Definition \ref{definitionofchoco},
 we have the following inequalities: 
 
\begin{figure}[!htb]
	\centering
\includegraphics[width=0.45\columnwidth,bb=0 0 260 169]{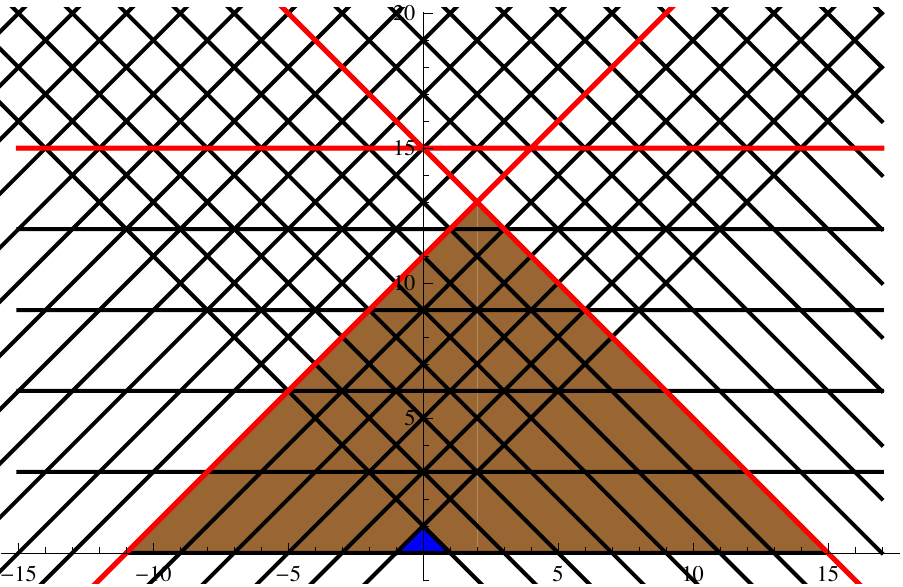}
\caption{Position $(5,4,7)$}
\label{demok32}
\end{figure}

\noindent
	$(a.2)$ $y \leq x+11$;
	
	\noindent
	$(b.2)$ $y \leq 15$;
	
	\noindent
	$(c.2)$ $y \leq -x+15$; and
	
	\noindent
	$(d.2)$ $y \geq 0$.
	
	\noindent
Lines $y = x+11$, $y = 15$ and $y = -x+15$	are represented by red lines in Figure \ref{demok32}.	
	Here, we can also omit inequality $(b.2)$ since the inequalities in $(a.2)$ and $(c.2)$ imply inequality $(b.2)$.

Next, we present an example that requires all inequalities in $(a)$, $(b)$, $(c)$, and $(d)$ of Definition \ref{definitionofchoco}.

\noindent
$(iii)$ Let $(u,v,w) =(4,2,7)$. Then, by the inequalities in $(a)$, $(b)$, $(c)$, $(d)$ of Definition \ref{definitionofchoco}, 
we have the following inequalities:

\begin{figure}[!htb]
	\centering
	\includegraphics[width=0.45\columnwidth,bb=0 0 260 115]{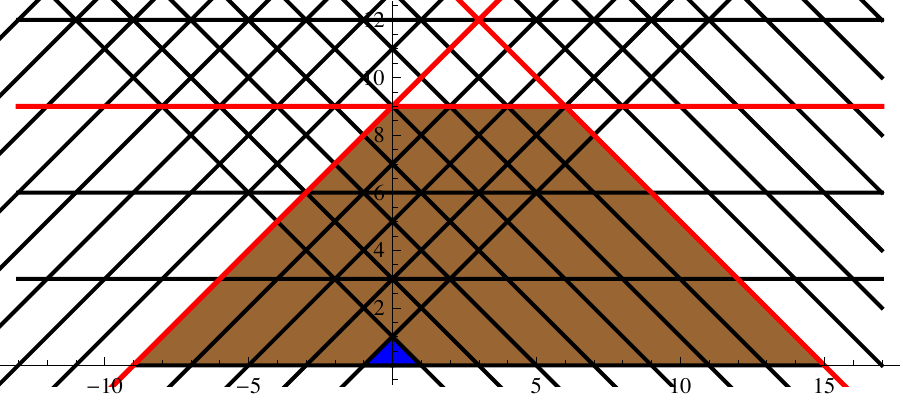}
	\caption{Position $(4,2,7)$}
	\label{demok33}
\end{figure}

\noindent
$(a.3)$ $y \leq x+9$;

\noindent
$(b.3)$ $y \leq 9$;

\noindent
$(c.3)$ $y \leq -x+15$; and

\noindent
$(d.3)$ $y \geq 0$.

\noindent
Lines $y = x+9$, $y=9$, and $y=-x+15$	are represented by red lines in Figure \ref{demok33}.	
Here, we cannot omit any inequalities.	
\end{exam}

\begin{exam}\label{examkequal7} Let $k =7$.
Here, we present examples of triangular chocolate bars of position $(u,v,w)$ for $u,v,w \in Z_{\ge 0}$
such that $kv \leq u + w$.

\noindent
$(i)$ Let $(u,v,w) =(9,2,9).$ Then, by the inequalities in $(a)$, $(b)$, $(c)$, and $(d)$ of Definition \ref{definitionofchoco}, 
we have the following inequalities: 

\begin{figure}[!htb]
	\centering
	\includegraphics[width=0.45\columnwidth,bb=0 0 260 148]{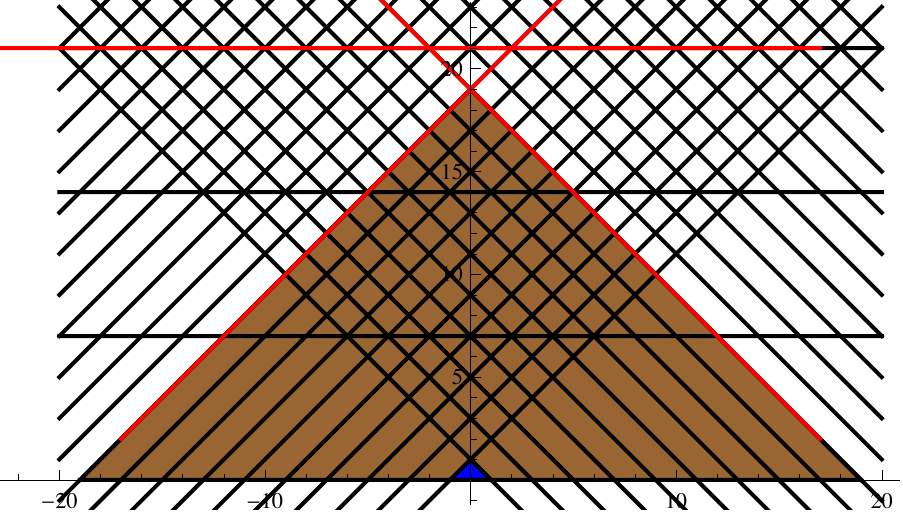}
	\caption{Position $(9,2,9)$}
	\label{demok51}
\end{figure}

\noindent
$(a.4)$ $y \leq x+19$;

\noindent
$(b.4)$ $y \leq 21$;

\noindent
$(c.4)$ $y \leq -x+19$; and 

\noindent
$(d.4)$ $y \geq 0$.

\noindent
Lines $y = x+19$, $y=21$, and $y=-x+19$	are represented by red lines in Figure \ref{demok33}.	
Here, we can omit inequality $(b.4)$ since the inequalities in $(a.4)$ and $(c.4)$ imply inequality $(b.4)$.		

\noindent
$(ii)$ Let $(u,v,w) =(6,2,8).$ Then, by the inequalities in $(a)$, $(b)$, $(c)$, and $(d)$ of Definition \ref{definitionofchoco}, we have the following inequalities:

\begin{figure}[!htb]
	\centering
	\includegraphics[width=0.45\columnwidth,bb=0 0 260 171]{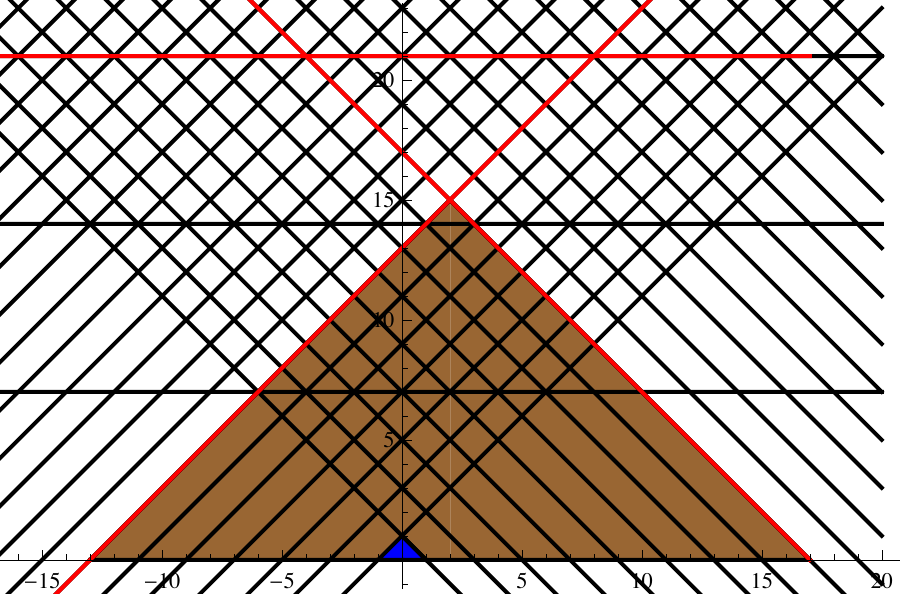}	
	\caption{Position $(6,2,8)$}
	\label{demok52}
\end{figure}

\noindent
$(a.5)$ $y \leq x+13$;

\noindent
$(b.5)$ $y \leq 21$;

\noindent
$(c.5)$ $y \leq -x+17$; and

\noindent
$(d.5)$ $y \geq 0$.

\noindent
Lines $y = x+13$, $y=21$, and $y=-x+17$	are represented by red lines in Figure \ref{demok52}.

\noindent	
Here, we can also omit inequality $(b.5)$ since the inequalities in $(a.5)$ and $(c.5)$ imply inequality $(b.5)$.

\noindent
$(iii)$ Let $(u,v,w) =(6,0,8)$. Then, by the inequalities in $(a)$, $(b)$, $(c)$, and $(d)$ of Definition \ref{definitionofchoco}, 
we have the following inequalities:	
	
\begin{figure}[!htb]
	\centering
\includegraphics[width=0.45\columnwidth,bb=0 0 260 83]{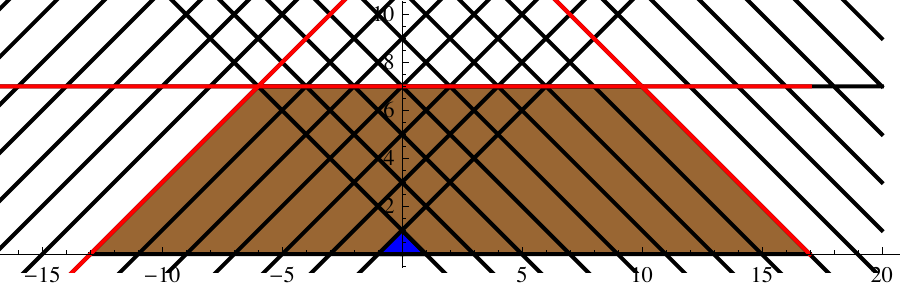}
\caption{Position $(6,0,8)$}
\label{demok53}	
\end{figure}

\noindent
$(a.6)$ $y \leq x+13$;

\noindent
$(b.6)$ $y \leq 7$;

\noindent
$(c.6)$ $y \leq -x+17$; and

\noindent
$(d.6)$ $y \geq 0$.	

Lines $y = x+13$, $y=7$, and $y=-x+17$	 are represented by red lines in Figure \ref{demok53}

\noindent		
Here, we cannot omit inequality $(b.6)$ since the inequalities in $(a.6)$ and $(c.6)$ do not imply inequality $(b.6)$.		
			
\end{exam}

In Definition \ref{definitionofchoco}, Example \ref{examkequal7}, and Example \ref{examkequal3}, we used the $x$ and $y$ coordinates to define the shape and position of the chocolate bar. However, in the remainder of this paper, we study chocolate bars and their positions without explicitly describing the coordinate system. 

\begin{exam}\label{howtocut}
Let $k=3$. Here, chocolates are presented with their positions $(x,y,z)$. 
These positions clearly satisfy the inequality 
$3y \leq x + z$, i.e., $y \leq \lfloor \frac{x+z}{3} \rfloor $, where $x,y,z$ are the first, second, and the third number of the position, respectively.
This example shows how to cut triangular chocolate bars.

\noindent
$(i)$ Suppose that we cut the chocolate bar in Figure \ref{chocok3747} to get the bar in Figure \ref{chocok3547} in the way demonstrated in Figure \ref{chocok3547cut}.
Starting with the chocolate bar in Figure \ref{chocok3747} that has the position $(x,y,z)=(7,4,7)$ and reducing $x$ to $5$, 
we have Figure \ref{chocok3547} with the position $(x,y,z)=(5,4,7)$.

\begin{figure}[!htb]
\begin{minipage}[!htb]{0.45\columnwidth}
	\centering
\includegraphics[width=0.7\columnwidth,bb=0 0 222 113]{chocok3747.pdf}
\caption{Position $(7,4,7)$}
\label{chocok3747}
\end{minipage}
\begin{minipage}[!htb]{0.45\columnwidth}
	\centering
\includegraphics[width=0.7\columnwidth,bb=0 0 193 98]{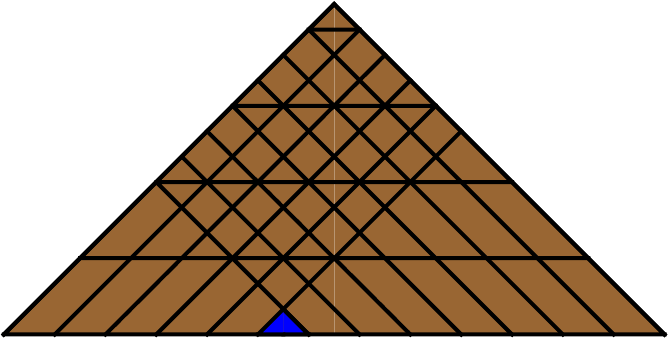}
\caption{Position $(5,4,7)$}
\label{chocok3547}
\end{minipage}
\end{figure}	

\begin{figure}[!htb]
	\centering
\includegraphics[width=0.45\columnwidth,bb=0 0 170 94]{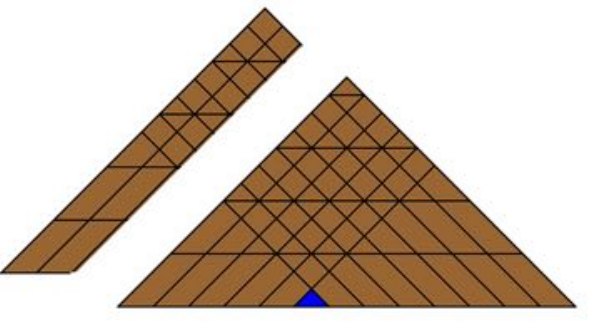}
\caption{cut the chocolate bar in Figure \ref{chocok3747}}	
\label{chocok3547cut}
\end{figure}

\noindent
$(ii)$ Suppose that we cut the chocolate bar in Figure \ref{chocok3747} to get the bar in Figure \ref{chocok3727} in the way demonstrated in Figure \ref{chococut727}.
Starting with the chocolate bar in Figure \ref{chocok3747} that has the position $(x,y,z)=(7,4,7)$ and reducing $y$ to $2$, 
 we have Figure \ref{chocok3727} with the position $(x,y,z)=(7,2,7)$.

\begin{figure}[!htb]
	\begin{minipage}[!htb]{0.45\columnwidth}
		\centering
\includegraphics[width=0.7\columnwidth,bb=0 0 260 84]{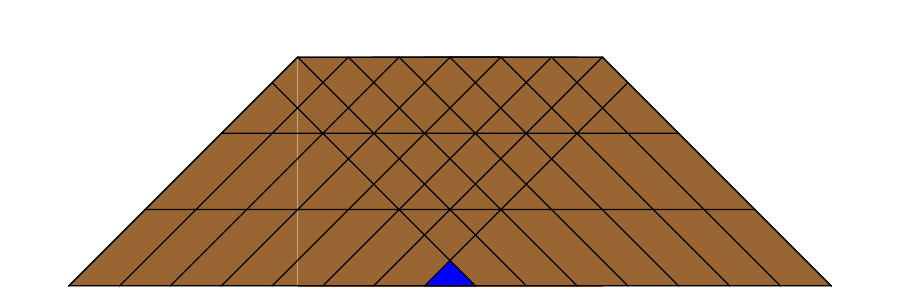}
\caption{Position $(7,2,7)$}
\label{chocok3727}
\end{minipage}
	\begin{minipage}[!htb]{0.45\columnwidth}
		\centering
\includegraphics[width=0.7\columnwidth,bb=0 0 170 94]{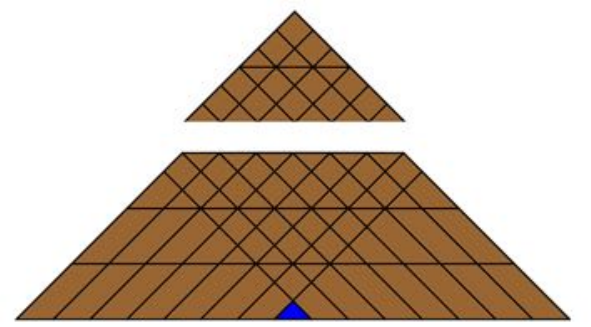}
\caption{Cut the chocolate bar in Figure \ref{chocok3727}}
\label{chococut727}
\end{minipage}
\end{figure}

\noindent
$(iii)$ Suppose that we cut the chocolate bar in Figure \ref{chocok3747} to get the chocolate in Figure \ref{chocok3732} in the way demonstrated in Figure \ref{chococut}.
Starting with the chocolate bar in Figure \ref{chocok3747} that has the position $(x,y,z)=(7,4,7)$ and reducing $z$ to $2$, 
the second number $y=4$ of the position will also be reduced to $y=3$. Therefore, we have Figure \ref{chocok3732} with the position $(x,y,z)=(7,3,2)$.

\begin{figure}[!htb]
	\begin{minipage}[!htb]{0.45\columnwidth}
		\centering
		\includegraphics[width=0.8\columnwidth,bb=0 0 260 84]{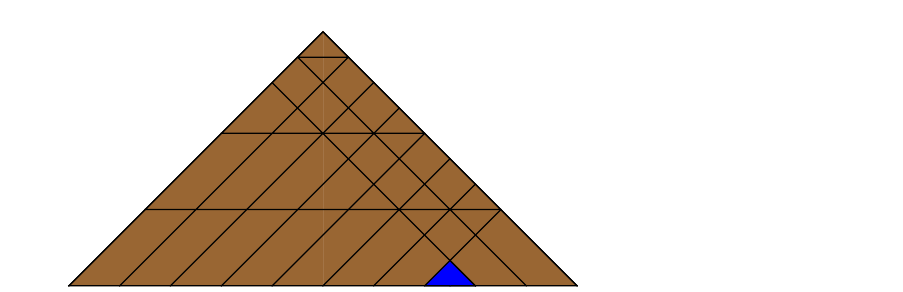}	
		\caption{Position $(7,3,2)$}
		\label{chocok3732}
	\end{minipage}
	\begin{minipage}[!htb]{0.45\columnwidth}
		\centering
		\includegraphics[width=0.8\columnwidth,bb=0 0 113 71]{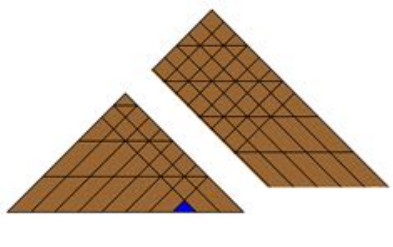}
		\caption{Cut the chocolate bar in Figure \ref{chocok3732}}
		\label{chococut}
	\end{minipage}
\end{figure}		
\end{exam}

Next, we are going to define the function $movek ((x, y, z))$ for position $(x, y, z)$ of triangular chocolate bars, where $movek ((x, y, z))$ is the set of all positions that can be reached from position $(x, y, z)$ in one step (directly).

\begin{defn}\label{defofmovek} Let $k$ be a natural number such that $k = 4m+3$ for some $m \in Z_{\geq0}$.
For $x,y,z \in Z_{\ge 0}$, we define $movek((x,y,z))=\{(u,y,z);u<x\} \cup \{(x,v,z);v<y\} \cup \{(x,y,w);w<z\} \cup \{(u,\min(y, \lfloor \frac{u+z}{k} \rfloor ),z);u<z\} $ 
\\$\cup \{(x,\min(y, \lfloor \frac{x+w}{k} \rfloor ),w);w<z\}$, where $u,v,w \in Z_{\ge 0}$.
\end{defn}

\begin{exam}
Here, we study the case of $k=3$.

\noindent
$(i)$ By Example \ref{howtocut}, $ (5,4,7) \in move3((7,4,7))$. 

\noindent
$(ii)$ Starting with the chocolate bar in Figure \ref{chocok3747} and reducing $z$ to $2$, 
 the second number of the position will be $\min(4, \lfloor \frac{7+2}{3} \rfloor )=\min(4,3)=3$. Therefore, we have the chocolate bar in Figure \ref{chocok3732} with the position $(7,3,2)$, and $ (7,3,2) \in move3((7,4,7))$. 
 
 \noindent
$(iii)$ Similarly, the chocolate bar with $(7,2,7)$ can be easily obtained by reducing the second number of the position to 2. Therefore, $ (7,2,7) \in move3((7,4,7))$.
\end{exam}

\section{Sequences of Three Functions}\label{studyofsequences}
Let $k = 4m+3$ for some $m \in Z_{\geq0}$. In this section, we study sequences constructed by three functions. These sequences of three functions will be used in Section \ref{relation1} to study the chocolate bar games that satisfy the inequality $y \leq \lfloor \frac{x+z}{k} \rfloor$.

\begin{defn}\label{definitionp1p2p3}
We define three functions 
${P_1}^k(x)=2x+2$, ${P_2}^k(x)=2x$, and ${P_3}^k(x)=2x+1-k$ for any $x \in Z_{\geq0}$.
\end{defn}
Note that ${P_1}^k(x) > {P_2}^k(x) > {P_3}^k(x)$ for any $x \in Z_{\geq0}$.

\begin{exam}\label{nonproperseq}
Let $k = 7$. Then, ${P_1}^7(x)=2x+2$, ${P_2}^7(x)=2x$, and ${P_3}^7(x)=2x+1-7=2x-6$.

\noindent
$(i)$ If we apply ${P_2}^7,{P_3}^7,{P_2}^7,{P_3}^7,{P_1}^7$ repeatedly to $2$, then 
we have 
$\{2,{P_2}^7(2)=4, {P_3}^7(4)=2,{P_2}^7(2)=4,
{P_3}^7 (4)=2, {P_1}^7 (2)=6\}$
$=\{2, 4,2,4,2,6\}$.\\
$(ii)$ 
 If we apply ${P_2}^7,{P_3}^7,{P_2}^7,{P_2}^7,{P_1}^7$ repeatedly to $2$, then we have 
$\{2,{P_2}^7(2)=4, {P_3}^7(4)=2,{P_2}^7(2)=4,
{P_2}^7(4)=8,{P_1}^7 (8)=18\}$
$=\{2, 4,2,4,8,18\}$.\\
$(iii)$ If we apply ${P_3}^7,{P_1}^7,{P_2}^7,{P_1}^7,{P_3}^7,{P_1}^7$ repeatedly to $2$, then we have
$\{2,{P_3}^7(2)=-2, {P_1}^7(-2)= -2,
{P_2}^7 (-2)=-4,{P_1}^7(-4)=-6,{P_3}^7(-6)=-18, {P_1}^7(-18)=-34\}$
$=\{2,-2,-2,-4,-6,-18,-34\}$.
\end{exam}

In Example \ref{nonproperseq}, we start with $2$ and make a sequence by repeatedly applying one of the three functions [${P_1}^k$, ${P_2}^k$, or ${P_3}^k$]. In this section, all sequences are constructed in this way. Note that these sentences are made by even numbers since we get only even numbers by applying ${P_1}^k$, ${P_2}^k$, or ${P_3}^k$ to $2$ repeatedly when $k$ is an odd number.

\begin{defn}\label{typeabc}
We define three types of sequences of integers:

\noindent
$(i)$ A sequence $\{b_0,b_1,b_2,...,b_n\}$ is said to be Type 1 if $0 \leq b_i < k$ for each $i = 0,1,2,...,n$.

\noindent
$(ii)$ A sequence $\{b_0,b_1,b_2,...,b_n\}$ is said to be Type 2 if there exists a non-negative number $t$ such that $0 \leq b_i < k$ for each $i = 0,1,2,...,t$ and $b_i \geq k$ for each $i = t+1,...,n$.

\noindent
$(iii)$ A sequence $\{b_0,b_1,b_2,...,b_n\}$ is said to be Type 3 if there exists a non-negative $t$ such that $0 \leq b_i < k$ for each $i = 0,1,2,...,t$ and $b_i < 0$ for each $i = t+1,...,n$.
\end{defn}

\begin{rem}
Sequences $(i)$, $(ii)$, and $(iii)$ in Example \ref{nonproperseq} are Type 1, Type 2, and Type 3 sequeneces, respectively.
\end{rem}

\begin{Lem}\label{therearethreetype}
Any sequence $\{b_0,b_1,b_2,...,b_n\}$ is Type 1, or Type 2, or Type 3, and nothing else.

\noindent
Regarding the criteria of these types, we have the following conditions:

\noindent
$(i)$ The sequence is Type 1 if and only if $0 \leq b_n < k$.

\noindent
$(ii)$ The sequence is Type 2 if and only if $ b_n \geq k$.

\noindent
$(iii)$ The sequence is Type 3 if and only if $b_n < 0$.
\end{Lem}

\begin{proof}
Suppose that $0 \leq b_i < k$ for each $i = 0,1,2,...,n$. This is Type 1.

\noindent
Suppose that there exists $t$ such that $0 \leq b_i < k$ for each $i = 0,1,2,...,t$ and $0 \leq b_{t+1} < k$ is not true. Then, we have
$b_{t+1} < 0$ or $b_{t+1} \geq k$. We study these cases separately.

\noindent
If $b_{t+1} < 0$, then $b_{t+1} \leq -2$ since $b_{t+1}$ is even. Then, by applying ${P_1}^k,{P_2}^k,{P_3}^k$ repeatedly to $b_{t+1}$, we get only a negative number, and the sequence we construct is Type 3.

\noindent
If $b_{t+1} \geq k$, by applying ${P_1}^k,{P_2}^k,{P_3}^k$ repeatedly to $b_{t+1}$, we get numbers bigger than $k$, and the sequence we construct is Type 2.

\noindent
Therefore, $\{b_0,b_1,b_2,...,b_n\}$ is Type 1, or Type 2, or Type 3, and nothing else.

\noindent
$(i)$ If the sequence is Type 1, then $0 \leq b_n < k$.

\noindent
Conversely, if $0 \leq b_n < k$, then by Definition \ref{typeabc}, this sequence is neither Type 2 nor Type 3. Therefore, this sequence is Type 1.

\noindent
$(ii)$ If the sequence is Type 2, then $ b_n \geq k$.

\noindent
Conversely, if $ b_n \geq k$, then by Definition \ref{typeabc}, this sequence is neither Type 1 nor Type 3. Therefore, this sequence is Type 2.

\noindent
$(iii)$ If the sequence is Type 3, then $b_n < 0$.

\noindent
Conversely, if $b_n < 0$, then by Definition \ref{typeabc}, this sequence is neither Type 1 nor Type 2. Therefore, this sequence is Type 3.
\end{proof}

\begin{Lem}\label{type2subsequence}
	Let $t,n \in Z_{\geq 0}$ such that $t \leq n$.
	If $\{b_0,b_1,b_2,...,b_t\}$ is Type 2, then $\{b_0,b_1,b_2,...,b_n\}$ is Type 2.
\end{Lem}

\begin{proof}
Since $b_t \geq k$, by Definition \ref{typeabc}, $\{b_0,b_1,b_2,...,b_n\}$ is not Type 1 or Type 3. Therefore, by Lemma \ref{therearethreetype}
$\{b_0,b_1,b_2,...,b_n\}$ is Type 2.
\end{proof}

\begin{Lem}\label{lemmaforwhichk}
For any even number $h$ with $0 \leq h < k$, we have the following statements: 

\noindent
$(i)$ If $0 \leq h \leq 2m$, then $0 \leq {P_2}^k(h) < {P_1}^k(h) < k$ and ${P_3}^k(h) < 0$.

\noindent
$(ii)$ If $2m+2 \leq h < k = 4m+3$, then $0 < {P_3}^k(h) < k$ and $k < {P_2}^k(h) < {P_1}^k(h)$.

\noindent
$(iii)$ If $k < {P_u}^k(h) $ for some $u \in \{1,2,3\}$, then $2m+2 \leq h $ and $u = 1$ or $u = 2$.

\noindent
Note that $h \neq 2m+1$ since $h$ is an even number.
\end{Lem}

\begin{proof}
$(i)$ Suppose that $0 \leq h \leq 2m$. Then, we have $0 \leq 2h < 2h+2 \leq 4m+2 < 4m+3 = k$, and by Definition \ref{definitionp1p2p3},
we have $0 \leq {P_2}^k(h) < {P_1}^k(h) < k$.

\noindent
We also have $2h+1-k \leq 4m+1-(4m+3) = -2 < 0$, and by Definition \ref{definitionp1p2p3}, ${P_3}^k(h) < 0$.

\noindent
$(ii)$ Suppose that $2m+2 \leq h < k = 4m+3$. Then, we have
$2m+2 \leq h \leq k-1$, and 
$2(2m+2)+1-k \leq 2h+1-k \leq 2(k-1)+1-k$.

\noindent
Therefore, we have $2 \leq 2h+1-k \leq k-1 < k$.
By Definition \ref{definitionp1p2p3}, $0 < {P_3}^k(h) < k$.

\noindent
We also have $k < 4m+4 \leq 2h < 2h+2$, and by Definition \ref{definitionp1p2p3}, $k < {P_2}^k(h) < {P_1}^k(h)$.

\noindent
$(iii)$ Suppose that 
\begin{align}\label{ksmallerthp}
k < {P_u}^k(h) 
\end{align}
 for some $u \in \{1,2,3\}$.
 By the contraposition of $(i)$ in this lemma, $2m+2 \leq h $. By $(ii)$ of this lemma, $0 < {P_3}^k(h) < k$ and $k < {P_2}^k(h) < {P_1}^k(h)$.
 Therefore, by the inequality in (\ref{ksmallerthp}), we have $u = 1$ or $u = 2$.
\end{proof}
	
\begin{Lem}\label{lemmaforwhichk2}
A sequence $\{b_0,b_1,b_2,...,b_n\}$ is Type 1 if and only if 
for each $i$ we have one of the following cases:

\noindent
$(i)$ If $0 \leq b_i \leq 2m$, then $b_{i+1} = {P_1}^k(b_i)$ or $b_{i+1} = {P_2}^k(b_i)$.

\noindent
$(ii)$ If $2m+2 \leq b_i < k = 4m+3$, then $b_{i+1} = {P_3}^k(b_i)$.

\noindent
Note that $b_i \neq 2m+1$ since $b_i$ is an even number.
\end{Lem}

\begin{proof}
This follows directly from Lemma \ref{lemmaforwhichk}.
\end{proof}
	
\begin{Lem}\label{lemmaforlonger}
If a sequence $\{b_0,b_1,b_2,...,b_t\}$ is Type 1 and $t < n$, then 
we can define $b_i$ for $i = t+1, t+2,...,n$ so that the sequence $\{b_0,b_1,b_2,...,b_n\}$ is Type 1.
\end{Lem}

\begin{proof}
We define $b_{i+1}$ for $i = t,...,n-1$, and we consider two cases.

\noindent
\underline{Case $(a)$}
 If $0 \leq b_i \leq 2m$, let $b_{i+1} = {P_1}^k(b_i)$ or $b_{i+1} = {P_2}^k(b_i)$. Then, by $(i)$ of Lemma \ref{lemmaforwhichk}, we have $0 \leq b_{i+1} \leq k$.
 
 \noindent
\underline{Case $(b)$} If $2m+2 \leq b_i < k$, let $b_{i+1} = {P_3}^k(b_i)$. Then, by $(ii)$ of Lemma \ref{lemmaforwhichk}, we have $0 \leq b_{i+1} \leq k$.
By Case $(a)$ and Case $(b)$, we construct a sequence $\{b_0,b_1,b_2,...,b_n\}$ such that $0 \leq b_j \leq k$ for $j = 0,1,...,n$.
This sequence is Type 1.
\end{proof}

\begin{Lem}\label{lemmaforwhichk3}
A sequence $\{c_0,c_1,c_2,...,c_n\}$ is Type 2 if and only if there is a sequence $\{b_0,b_1,b_2,...,b_n\}$ that satisfies one of the following cases:

\noindent
$(i)$ $\{b_0,b_1,b_2,...,b_n\}$ is Type 1.

\noindent
$(ii)$ There is a non-negative number $j$ such that $b_t = c_t$ for $t = 0,1,...,j, 2m+2 \leq c_j=b_j < k$, $b_{j+1} = {P_3}^k(b_j)= {P_3}^k(c_j)$ and $c_{j+1} = {P_1}^k(c_j)$ or ${P_2}^k(c_j)$.
\end{Lem}

\begin{proof}
	Let $\{c_0,c_1,c_2,...,c_n\}$ be Type 2. Then, there is a non-negative number $j$ such that $0 \leq c_t < k$ for $t = 0,1,...,j$ and 
	$k < c_{j+1} $. By $(iii)$ of Lemma \ref{lemmaforwhichk}, we have 
	 $2m+2 \leq c_j < k$ and $c_{j+1} ={P_2}^k(c_j)$ or $ c_{j+1}={P_1}^k(c_j)$.
	 Let $b_i = c_i$ for $i = 0,...,j$. Then, $\{b_0,b_1,b_2,...,b_j\}$ is Type 1. By Lemma \ref{lemmaforlonger},
	 we can construct a sequence $\{b_0,b_1,b_2,...,b_n\}$ that is Type 1.
	 
	 \noindent
	Conversely, suppose that 
	there is a sequence $\{b_0,b_1,b_2,...,b_n\}$ that satisfies $(i)$ and $(ii)$ of this Lemma.
	Then, $0 \leq c_t < k$ for $t = 0,1,...,j$ and 
	$2m+2 \leq c_j=b_j < k$, and $c_{j+1} = {P_1}^k(c_j)$ or ${P_2}^k(c_j)$.
	Then, by $(ii)$ of Lemma \ref{lemmaforwhichk}, 
	$c_{j+1} ={P_2}^k(c_j)>k$ or $ c_{j+1}={P_1}^k(c_j)>k$.
	Therefore, $\{c_0,c_1,c_2,...,c_n\}$ is Type 2.
\end{proof}

\section{Sequence Generated by {\boldmath $x,y,z$}}\label{sequencemadebyxyz}

Let $x,y,z \in Z_{\geq0}$ in base $2$, so 
\begin{align}\label{xyzinbase}
	x = \sum_{i=0}^n x_i 2^i , y = \sum_{i=0}^n y_i 2^i \text{ and } z = \sum_{i=0}^n z_i 2^i \text{ with } x_i,y_i,z_i \in \{0,1\}.
\end{align}

Throughout this section, we suppose that

\begin{align}\label{defnofsequencesn11}
	x \oplus y \oplus z = 0
\end{align}
and

\begin{align}\label{defnofsequencesn21}
	x_n=z_n=1 \text{ \ and \ } y_n=0.
\end{align}

By (\ref{defnofsequencesn11}),
$x_i + y_i + z_i = 0 \ (mod \ 2)$, and hence 

\begin{align}\label{defnofsj00000}
	(x_i,y_i,z_i) = (1,0,1) \text{ or } (0,0,0) \text{ or } (1,1,0) \text{ or } (0,1,1).
\end{align}

\begin{defn}\label{defnofsequencesj}
We define a sequence of non-negative integers 
$s_0,s_1,...,s_n$ by 
\begin{align}\label{defnofsj00}
s_j=\sum\limits_{i = n-j}^n {({x_i}+{z_i}-k{y_i})} {2^{i+j-n}}
\end{align}
for $j = 0,1,2,...,n$.

\noindent
This sequence $s_0,s_1,...,s_n$ is said to be generated by $x,y,z$.
\end{defn}

\begin{Lem}\label{valueofs0sn}
	By (\ref{defnofsequencesn21}) and (\ref{defnofsj00}),
	\begin{align}
		&	s_0=\sum\limits_{i = n}^n {({x_i}+{z_i}-k{y_i})} {2^{i-n}} = x_n + z_n -ky_n = 2, \label{sequenceforj0} \\
		&	s_1=\sum\limits_{i = n-1}^n {({x_i}+{z_i}-k{y_i})} {2^{i+1-n}} =(x_{n-1} + z_{n-1} -ky_{n-1}) 2^0 \nonumber\\
		& + (x_n + z_n -ky_n)2^1 \label{sequenceforj1} \\
		&\text{and } \nonumber \\
		&	s_n=\sum\limits_{i = 0}^n {({x_i}+{z_i}-k{y_i})} {2^{i}} = x + z -ky. \label{xzkyrelation}
	\end{align}
\end{Lem}

\begin{proof}
This lemma follows directly from Definition \ref{defnofsequencesj}.
\end{proof}

\begin{Lem}\label{threesequence}
$(i)$ If $(x_{n-j-1},y_{n-j-1},z_{n-j-1})=(1,0,1)$, then 
\begin{align}\label{caseof1}
	s_{j+1}=2s_j + 2 = {P_1}^k(s_j).
\end{align}
$(ii)$ If $(x_{n-j-1},y_{n-j-1},z_{n-j-1})=(0,0,0)$, then
\begin{align}\label{caseof2}
	s_{j+1}=2s_j ={P_2}^k(s_j).
\end{align}
$(iii)$ If $(x_{n-j-1},y_{n-j-1},z_{n-j-1})=(1,1,0)$ or $(0,1,1)$, then
\begin{align}\label{caseof3}
	s_{j+1}= 2s_j+1-k={P_3}^k(s_j).
\end{align}
\end{Lem}

\begin{proof}
By (\ref{defnofsj00}),
\begin{align}
	s_{j+1}=& \sum\limits_{i = n-j-1}^n {({x_i}+{z_i}-k{y_i})} {2^{i+j+1-n}}  \nonumber \\
	= & 2(\sum\limits_{i = n-j}^n {({x_i}+{z_i}-k{y_i})} {2^{i+j-n}})+x_{n-j-1}+z_{n-j-1}-k y_{n-j-1} \nonumber \\
	=& 2s_j+x_{n-j-1}+z_{n-j-1}-k y_{n-j-1}.  \label{defsjplus}
\end{align}

By (\ref{defsjplus}), we have $(i)$, $(ii)$, or $(iii)$.
\end{proof}

By Lemma \ref{threesequence}, $s_0,s_1,...,s_n$
is the sequence generated by applying ${P_1}^k$, ${P_2}^k$ and ${P_3}^k$
repeatedly to 2.

\section{Relationship Between the Inequality and the Sequence}\label{relation1}

\begin{Lem}\label{criteriaofineqa}
	For $x,y,z \in Z_{\geq0}$, the following conditions hold:
	
	\noindent
	$(i)$ $0 \leq x+z-ky < k$ if and only if $ y = \lfloor \frac{x+z}{k} \rfloor$;
	
	\noindent
	$(ii)$ $k \leq x+z-ky $ if and only if $ y < \lfloor \frac{x+z}{k} \rfloor$; and
	
	\noindent
	$(iii)$ $ x+z-ky < 0$ if and only if $ y > \lfloor \frac{x+z}{k} \rfloor$.
\end{Lem}

\begin{proof}
	$(i)$ $0 \leq x+z-ky < k$ if and only if $ky \leq x+z < ky +k$ 
	if and only if $y \leq \frac{x+z}{k} < y+1$ if and only if 
	$ y = \lfloor \frac{x+z}{k} \rfloor$.\\
	Similarly, we prove $(ii)$ and $(iii)$, and we finish the proof.
\end{proof}

\begin{Lem}\label{lemmaforwhichk20}
	Let $x,y,z \in Z_{\geq0}$ such that 
\begin{align}\label{xyzoplus0}
x \oplus y \oplus z = 0,
\end{align}
$x_n = z_n = 1$, and $y_n = 0$,
and let $s_0,s_1,...,s_n$ be the sequence generated from $x,y,z$.
Then, $y = \lfloor \frac{x+z}{k} \rfloor $ if and only if 
$s_0,s_1,...,s_n$ is Type 1. For each $j$,
the following conditions hold:

\noindent
$(a)$ If $s_j \leq 2m$, then we have $(a.1)$ or $(a.2)$:

\noindent
$(a.1)$
$(x_{n-j-1},y_{n-j-1},z_{n-j-1})$$=(1,0,1)$ and 
\begin{align}
s_{j+1}={P_1}^k(s_j).
\end{align}
$(a.2)$ $(x_{n-j-1},y_{n-j-1},z_{n-j-1})$$=(0,0,0)$ and 
\begin{align}
s_{j+1}=2s_j ={P_2}^k(s_j).
\end{align}
$(b)$ If $s_j \geq 2m+2$, then $(x_{n-j-1},y_{n-j-1},z_{n-j-1})$ $=(1,1,0)$ or $(0,1,1)$, 
and 
\begin{align}
s_{j+1}= 2s_j+1-k={P_3}^k(s_j).
\end{align}
\end{Lem}

\begin{proof}
Suppose that $y = \lfloor \frac{x+z}{k} \rfloor$.
Then, by Lemma \ref{criteriaofineqa} and (\ref{xzkyrelation}), we have $0 \leq s_n=x+z-ky < k$, and by $(i)$ of Lemma \ref{therearethreetype}, the sentence $s_0,s_1,...,s_n$ is Type 1.
Conversely, if the sentence $s_0,s_1,...,s_n$ is Type 1, then by Lemma \ref{criteriaofineqa}, (\ref{xzkyrelation}), and Lemma \ref{therearethreetype}, 
$y = \lfloor \frac{x+z}{k} \rfloor$.
By Lemma \ref{lemmaforwhichk2}, for each $j$,
we have the following $(a)$ or $(b)$.

\noindent
$(a)$ If $s_j \leq 2m$, then we have $(a.1)$ or $(a.2)$:

\noindent
$(a.1)$
If
\begin{align}
s_{j+1}={P_1}^k(s_j),
\end{align}
then by (\ref{caseof1}), $(x_{n-j-1},y_{n-j-1},z_{n-j-1})$$=(1,0,1)$.\\
$(a.2)$
If
\begin{align}
s_{j+1}=2s_j ={P_2}^k(s_j),
\end{align}
then by (\ref{caseof2}), $(x_{n-j-1},y_{n-j-1},z_{n-j-1})$ $=(0,0,0)$.\\
$(b)$ If $s_j \geq 2m+2$, 
\begin{align}
s_{j+1}= 2s_j+1-k={P_3}^k(s_j)
\end{align}
and $(x_{n-j-1},y_{n-j-1},z_{n-j-1})$ $=(1,1,0)$ or $(0,1,1)$.
\end{proof}

\begin{Lem}\label{nimsum0lemma}
Let $x \in Z_{\geq0}\ $, and let $y_i,z_i \in \{0,1\}$ such that $x_i\oplus y_i\oplus z_i=0 $ for $i = n, n-1,...,n-t$ for a fixed natural number $t$, and $x_n=1, y_n = 0, z_n=1.$

We define the sequence $s_j, j = 0,1,2,...,t$ by
\begin{align}\label{defnofsj002}
s_j=\sum\limits_{i = n-j}^n {({x_i}+{z_i}-k{y_i})} {2^{i+j-n}}.
\end{align}
Suppose that $0 \leq s_i < k$ for $i=1,2,...,t$. Then, there exist unique $y_i,z_i \in \{0,1\}$ \\
for $i = n-t-1, n-t-2,...,0$ such that 
\begin{align}\label{nimsum0xyz}
x\oplus y\oplus z=0 \text{ and } y = \lfloor \frac{x+z}{k} \rfloor,
\end{align}
where $y = \sum\limits_{i = 0}^n {{y_i}} {2^i}$ and $z = \sum\limits_{i = 0}^n {{z_i}} {2^i}$.
\end{Lem}

\begin{proof}
	Let $x \in Z_{\geq0}\ $.
We write $x$ in base $2$, so \\
 \begin{align}\label{alignfxyz}
x = \sum\limits_{i = 0}^n {{x_i}} {2^i},
\end{align}

We define sequences $y_i,z_i \in \{0,1\}$ for $i = n-t-1, n-t-2,...,0$, and we define
$s_j, j = t+1,..., n-1,n$ using 
(\ref{defnofsj002}) step by step.

First, we define $y_{n-t-1}, z_{n-t-1}$ and $s_{t+1}$.\\
We have the following two cases:\\
\underline{Case $(a)$}
Suppose that 
\begin{align}\label{casea12}
0 \leq s_t \leq 2m.
\end{align}
Then, we have Subcase $(a.1)$ and Subcase $(a.2)$:\\
\underline{Subcase $(a,1)$}
If $x_{n-t-1}=1$, then let $(x_{n-t-1},y_{n-t-1},z_{n-t-1}) = (1,0,1)$. Then, by (\ref{caseof1}), 
$s_{t+1} = {P_1}^n(s_t)$,
and by Lemma \ref{lemmaforwhichk} and (\ref{casea12}), we have $0 \leq s_{t+1} < k$.\\
\underline{Subcase $(a,2)$}
If $x_{n-t-1}=0$, then let $(x_{n-t-1},y_{n-t-1},z_{n-t-1}) = (0,0,0)$. Then, by (\ref{caseof2}), 
$s_{t+1} = $${P_2}^n(s_t)$, and by 
Lemma \ref{lemmaforwhichk} and (\ref{casea12}), we have $0 \leq s_{t+1} < k$.\\
\underline{Case $(b)$}
Suppose that 
\begin{align}\label{caseb12}
2m+2 \leq s_t < k
\end{align}
Then, we have Subcase $(b.1)$ and Subcase $(b.2)$:\\
\underline{Subcase $(b,1)$}
If $x_{n-t-1}=1$, then let $(x_{n-t-1},y_{n-t-1},z_{n-t-1}) = (1,1,0)$. Then, by (\ref{caseof3}), $s_{t+1} = $${P_3}^n(s_t)$, and by 
Lemma \ref{lemmaforwhichk} and (\ref{caseb12}), we have $0 \leq s_{t+1} < k$.\\
\underline{Subcase $(b,2)$}
If $x_{n-t-1}=0$, then let $(x_{n-t-1},y_{n-t-1},z_{n-t-1}) = (0,1,1)$. Then, by (\ref{caseof3}), $s_{t+1} = $${P_3}^n(s_t)$, and by 
Lemma \ref{lemmaforwhichk} and (\ref{caseb12}), we have $0 \leq s_{t+1} < k$.\\
By Case $(a)$ and Case $(b)$, we have 
\begin{align}\label{nimsumxyzj1}
	x_{n-t-1} \oplus y_{n-t-1} \oplus z_{n-t-1} =0 
\end{align}
and
\begin{align}\label{nimsumxyzj12}
0 \leq s_{t+1} < k. 
\end{align}
Clearly, $y_{n-t-1}, z_{n-t-1}$ are unique, non-negative integers that satisfy (\ref{nimsumxyzj1}) and (\ref{nimsumxyzj12}) when $x$ is a given non-negative integer.

Next, we define $y_{n-t-2}, z_{n-t-2}$, and $s_{t+2}$ using a method very similar to that used in Case $(a)$ and Case $(b)$.
In this way, we construct sequences $y_i, i=n-t-3, n-t-4, ...,0$, $z_i, i = n-t-3, n-t-4, ...,0$ and $s_j, j = t+3,t+4,...,n$ such that 
\begin{align}\label{nimsum0xyz1}
	x_i \oplus y_i \oplus z_i =0 
\end{align}
and $0 \leq s_j < k$. Then, the sequence $s_n, s_{n-1},...,s_2,s_1,s_0$ is Type 1, and 
by Lemma \ref{lemmaforwhichk20}, $y = \lfloor \frac{x+z}{k} \rfloor $, where $y = \sum\limits_{i = 0}^n {{y_i}} {2^i}$ and $z = \sum\limits_{i = 0}^n {{z_i}} {2^i}$.
The uniqueness of $y$ and $z$ is clear from the procedure used to determine the value of $y_i$ and $z_i$.
\end{proof}

\begin{Lem}\label{lammeofkyxz12}
Let $x,y,z,v,w \in Z_{\geq0}\ $ such that 
\begin{align}\label{nimsuminequaxyz}
x\oplus y\oplus z=0 \text{ and } y = \lfloor \frac{x+z}{k} \rfloor
\end{align}
and
\begin{align}\label{nimsumineuvw2}
x\oplus v\oplus w=0 \text{ and } v < \lfloor \frac{x+w}{k} \rfloor.
\end{align}
Then, there exists $t \in Z_{\geq0}$ such that 
$y_{n-j} = v_{n-j}$ and $z_{n-j} = w_{n-j}$ for $j = 0,1,2,...,t$ and $y_{n-t-1} > v_{n-t-1}$.
In particular, $y > v$.
\end{Lem}
\begin{proof}

We define a sequences of non-negative integers 
$r_0,r_1,...,r_n$ by 

\begin{align}\label{defnofsjj2}
	r_j=\sum\limits_{i = n-j}^n {({x_i}+{w_i}-k{v_i})} {2^{i+j-n}}.
\end{align}

By $(ii)$ of Lemma \ref{therearethreetype}, $(ii)$ of Lemma \ref{criteriaofineqa}, (\ref{xzkyrelation}), and (\ref{nimsumineuvw2}), the sequence $r_0,r_1,...,r_n$ is Type 2. Hence, 
there exists a natural number $t$ such that $0 \leq r_j < k$ for $j = 0,1,2,...,t$ and $k < r_{t+1}$.
By $(ii)$ of Lemma \ref{lemmaforwhichk}, $2m+2 \leq r_t < k$ and $r_{t+1} = {P_2}^k(r_t)$ or $r_{t+1} = {P_1}^k(r_t)$.

Therefore, by Lemma \ref{lemmaforwhichk}, (\ref{caseof1}), and (\ref{caseof2}), we have two cases:\\
Case $(a)$ Suppose that $x_{n-t-1} = 1$. Then, $r_{t+1} = {P_2}^k(r_t)>k$ and \\
$(x_{n-t-1},v_{n-t-1},w_{n-t-1})$$=(1,0,1)$.\\
Let $(v^{\prime}_{n-i},w^{\prime}_{n-i})$$=(v_{n-i},w_{n-i})$ for $i = 0,1,...,t$ and 
 $(v^{\prime}_{n-t-1},w^{\prime}_{n-t-1})$$=(1,0)$. Then,
 $x_{n-j} \oplus v^{\prime}_{n-j} \oplus w^{\prime}_{n-j} = 0$ for $j = 0,1,2,...,t+1$.
 By Lemma \ref{nimsum0lemma}, there exist
$v^{\prime}_i,w^{\prime}_i \in \{0,1\}$ for $i = n-t-1, n-t-2,...,0$ such that 
\begin{align}\label{nimsum0xyzbb}
	x\oplus v^{\prime}\oplus z=0 \text{ and } v^{\prime} = \lfloor \frac{x+w^{\prime}}{k} \rfloor,
\end{align}
where $v^{\prime}= \sum\limits_{i = 0}^n {{v^{\prime}_i}} {2^i}$ and $w^{\prime} = \sum\limits_{i = 0}^n {{w^{\prime}_i}} {2^i}$.
By Lemma \ref{nimsum0lemma}, 
 $v^{\prime}, w^{\prime}$ are unique non-negative integers that satisfy (\ref{nimsum0xyzbb}). Hence,
we have $y = v^{\prime}$ and $z = w^{\prime}$. Then, $y_{n-j} = v_{n-j}$ and $z_{n-j} = w_{n-j}$ for $j = 0,1,2,...,t$ and $y_{n-t-1} > v_{n-t-1}$.
In particular, $y > v$.

Case $(b)$ Suppose that $x_{n-t-1} = 0$. Then, $r_{t+1} = {P_1}^k(r_t)>k$ and \\
$(x_{n-t-1},v_{n-t-1},w_{n-t-1})$$=(0,0,0)$.\\
Let $(v^{\prime}_{n-t-1},w^{\prime}_{n-t-1})$$=(1,1)$.
Using a similar method to that used in Case $(a)$, we finish the proof.
\end{proof}

\begin{thm}\label{theoremfor4mplus3a}
Suppose that $x \oplus y \oplus z=0$ and $y \le \lfloor {\frac{x+z}{k}} \rfloor$. Then, the following conditions hold:\\
$(i)$ $u \oplus y \oplus z \ne 0$ for any $u\in Z_{\geq 0}$ with $u<x$;\\
$(ii)$ $x \oplus v \oplus z \ne 0$ for any $v\in Z_{\geq 0}$ with $v<y$;\\
$(iii)$ $x \oplus y \oplus w \ne 0$ for any $w\in Z_{\geq 0}$ with $w<z$;\\
$(iv)$ $x \oplus v \oplus w \ne 0$ for any $v,w\in Z_{\geq 0}$ with $v<y,w<z$ and $v = \lfloor {\frac{x+w}{k}} \rfloor$; and\\
$(v)$ $u \oplus v \oplus z \ne 0$ for some $u,v \in {Z_{ \ge 0}}$ with $u<x,v<y$ and $v=\lfloor {\frac{u+z}{k}}\rfloor$.
\end{thm}

\begin{proof}
Here, $(i)$, $(ii)$, and $(iii)$ come directly from definition of nim-sum $\oplus$.\\
 We suppose that $x \oplus v \oplus w = 0$ and $v = \lfloor {\frac{x+w}{k}} \rfloor$ for some $v,w \in {Z_{ \ge 0}}$ with $v<y,w<z$. 
If $y < \lfloor {\frac{x+z}{k}} \rfloor $, then by Lemma \ref{lammeofkyxz12}, we have $y<v$. This contradicts the fact $v<y$.
If $y = \lfloor {\frac{x+z}{k}} \rfloor $, then by Lemma \ref{nimsum0lemma},
 we have $y=v$. This contradicts the fact $v<y$. Therefore, $x \oplus v \oplus w \ne 0$, and we have $(iv)$\\
$(v)$ This case can be proved using the same method as that used in $(iv)$.
\end{proof}

\begin{Lem}\label{lemmaforimportcase}
Suppose that 
\begin{align}\label{conditionp1}
y \le 	\lfloor {\frac{x+z}{k}} \rfloor
\end{align}
and
\begin{align}\label{conditionp2}
x_i +y_i +z_i =0\ (mod\ 2) \ \text{for } i=m+1,m+2,...,n.
\end{align}
We define $s_j$ for $j=0,1,2,...,n-m-1$ by
 \begin{align}\label{defnofsjj1b}
s_j=\sum\limits_{i = n-j}^n {({x_i}+{z_i}-k{y_i})} {2^{i+j-n}}.
\end{align}
Then, $s_0,s_1,...,s_{n-m-1}$ is Type 1 or Type 2.
\end{Lem}

\begin{proof}
Let $(u_i,v_i,w_i)=(x_i,y_i,z_i)$ for $i=m+1,m+2,...,n$ and 
$(u_i,v_i,w_i)=(0,0,0)$ for $i=0,1,2,...,m$.
Then, by (\ref{conditionp1}), we have $\lfloor \frac{v}{2^{m+1}} \rfloor \le \frac{\lfloor \frac{u}{2^{m+1}} \rfloor+\lfloor \frac{w}{2^{m+1}} \rfloor}{k}$. Multiplying both sides of the inequality by $2^{m+1}$, we get
 \begin{align}\label{ineqvux}
v \le \lfloor {\frac{u+w}{k}} \rfloor .
 \end{align}

We define $r_j$ for $j=0,1,2,...,n$ by
 \begin{align}\label{defnofsjj1}
r_j=\sum\limits_{i = n-j}^n {({u_i}+{w_i}-k{v_i})} {2^{i+j-n}}.
\end{align}
Then, by (\ref{ineqvux}), Lemma \ref{criteriaofineqa}, Lemma \ref{valueofs0sn}, and Lemma \ref{therearethreetype},
 $r_0,r_1,...,r_n$ is Type 1 or Type 2. Since $s_j = r_j$ for $j=0,1,...,n-m-1$, we prove this lemma.
\end{proof}

\begin{thm}\label{theoremfor4mplus3b}
Suppose that $x \oplus y \oplus z \ne 0$ and $y \le \lfloor {\frac{x+z}{k}} \rfloor$.\\
Then, at least one of the following statements is true:\\
$(i)$ $x \oplus y \oplus w = 0$ for some $w \in {Z_{ \ge 0}}$ with $w<z$ and $y \le \lfloor \frac{x+w}{k} \rfloor$;\\
$(ii)$ $x \oplus v \oplus z = 0$ for some $v \in {Z_{ \ge 0}}$ with $v < y \le \lfloor \frac{x+z}{k} \rfloor$;\\
$(iii)$ $u \oplus y \oplus z = 0$ for some $u \in {Z_{ \ge 0}}$ with $u<x$ and $y \le \lfloor \frac{u+z}{k} \rfloor$;\\
$(iv)$ $x \oplus v \oplus w = 0$ for some $v,w \in {Z_{ \ge 0}}$ with $v<y,w<z$ and $v=\lfloor {\frac{x+w}{k}}\rfloor$; or\\
$(v)$ $u \oplus v \oplus z = 0$ for some $u,v \in {Z_{ \ge 0}}$ with $u<x,v<y$ and $v=\lfloor {\frac{u+z}{k}}\rfloor$.
\end{thm}

\begin{proof}
Suppose that $x_i +y_i +z_i =0\ (mod\ 2)$ for $i=n, n-1,...,n-m+1$ and $x_{n-m} +y_{n-m} +z_{n-m} \neq 0\ (mod\ 2)$.
We consider three cases:\\
\underline{Case $(a)$} Suppose that $z_{n-m} =1$ and $x_{n-m} =y_{n-m} =0$. 

 We define $s_j $ for $j = 0,1,2,...,m$ by 
 \begin{align}\label{defnofsjj10}
 	s_j=\sum\limits_{i = n-j}^n {({x_i}+{z_i}-k{y_i})} {2^{i+j-n}}.
 \end{align}
By Lemma \ref{lemmaforimportcase}, the sequence $\{s_0,s_1,...,s_{m-1}\}$
is Type 1 or Type 2.
We have two subcases:\\
\underline{Subcase$(a.1)$} Suppose that the sequence $s_j $ for $j = 0,1,2,...,m-1$ is Type 2. 
Let $w_i = x_i+y_i \ (mod \ 2)$ for $i=n-m, n-m-1,...,1,0$ and $w_i = z_i$ for $i=n,n-1,...,n-m+1$.
We define 
 $r_j $ for $j = 0,1,2,...,n$ by 
 \begin{align}\label{defnofsjj11}
 	r_j=\sum\limits_{i = n-j}^n {({x_i}+{w_i}-k{y_i})} {2^{i+j-n}}.
 \end{align}
Since the sequence $\{s_0,s_1,...,s_{m-1}\}$ is Type 2 and $r_j=s_j$ for $j=0,1,...,m-1$,
by Lemma \ref{type2subsequence}, the sequence $\{r_0,r_1,...,r_{n}\}$ is Type 2.
Therefore, we have $y \le \lfloor {\frac{x+w}{k}} \rfloor$. Then, we have $(i)$.

\noindent
\underline{Subcase $(a.2)$} Suppose that the sequence $s_j $ for $j = 0,1,2,...,m-1$ is Type 1. 
Let $w_i = z_i$ for $i=n,n-1,...,n-m+1$ and $w_{n-m} = 0$.
We define $r_j $ for $j = 0,1,2,...,m$ by 
 \begin{align}\label{defnofsjj12}
 	r_j=\sum\limits_{i = n-j}^n {({x_i}+{w_i}-k{y_i})} {2^{i+j-n}}.
 \end{align}
 Then, we have two subsubcases:
 
 \noindent
 \underline{Subsubcase $(a.2.1)$} Suppose that the sequence $r_j $ for $j = 0,1,2,...,m$ is Type 2. 
 Then,
let $w_i = x_i+y_i \ (mod \ 2)$ for $i=n-m-1,...,1,0$.
We define 
$r_j $ for $j = m+1,m+2,...,n$ by 
\begin{align}\label{defnofsjj13}
	r_j=\sum\limits_{i = n-j}^n {({x_i}+{w_i}-k{y_i})} {2^{i+j-n}}.
\end{align}
Since the sequence $\{r_0,r_1,...,r_{m}\}$ is Type 2, 
by Lemma \ref{type2subsequence}, the sequence $\{r_0,r_1,...,r_{n}\}$ is Type 2.
Therefore, we have $y \le \lfloor {\frac{x+w}{k}} \rfloor$. Then, we have $(i)$.

\noindent
\underline{Subsubcase $(a.2.2)$} Suppose that the sequence $r_j $ for $j = 0,1,2,...,m$ is Type 1. 
By Lemma \ref{nimsum0lemma}, there exist unique $v_i,w_i \in \{0,1\}$ for $i = n-m-1,...,0$ such that 
\begin{align}\label{nimsum0xyz2}
	x\oplus u\oplus w=0 \text{ and } u = \lfloor \frac{x+w}{k} \rfloor,
\end{align}
 Then, we have $(iv)$.
 
 \noindent
 \underline{Case $(b)$} Suppose that $x_{n-m} =1$ and $z_{n-m} =y_{n-m} =0$. We can use the same method used in Case $(a)$.
 
 \noindent
 \underline{Case $(c)$} Suppose that $y_{n-m} =1$ and $z_{n-m} =x_{n-m} =0$. Let $v_i = x_i+y_i$ for $i=n-m, n-m-1, n-m-2,...,1,0$.
 Then, we have $x \oplus v \oplus z = 0$ and $v < y \le \lfloor \frac{x+z}{k} \rfloor$, and this is $(ii)$ of this lemma.
\end{proof}

\begin{defn}\label{defofABkybiggerxz}
Here, we define sets of positions of chocolate bars.\\
Let $A_{k}=\{(x,y,z);x,y,z\in Z_{\geq 0},y \leq \lfloor \frac{x+z}{k} \rfloor$ and $x\oplus y \oplus z=0\}$, $B_{k}=\{(x,y,z);x,y,z\in Z_{\geq 0},y \leq \lfloor \frac{x+z}{k} \rfloor$, and $x\oplus y \oplus z\neq 0\}$.
\end{defn}

\begin{thm}\label{theoremforkyxzchocol}
Let $k=4m+3$. Then, $A_k$ and $B_k$ are, respectively, the sets of $\mathcal{P}$-positions
and $\mathcal{N}$-positions
of the chocolate bar game that satisfies the inequality $y \leq \lfloor \frac{x+z}{k} \rfloor$.
\end{thm}

To prove Theorem \ref{theoremforkyxzchocol}, we need two additional theorems.
First, we prove that starting with an element of $A_k$, any move leads to an element of $B_k$.

\begin{rem}\label{remarkforot}
For $k=4m+1$, a conjecture for the generic formula for $\mathcal{P}$-positions is presented in Subsection \ref{4mplus1}.
When $k$ is an even number, there seems to be any formula for $\mathcal{P}$-positions. See Subsection \ref{evencase}.
\end{rem}

\begin{thm}\label{thforkfrAtoB}
For any $(x,y,z) \in A_k$, we have $movek((x,y,z)) \subset B_k$.
\end{thm}

\begin{proof}
Let $(x,y,z) \in A_k$. Then, we have
\begin{align}
 x \oplus y \oplus z \oplus =0
 \end{align}
 and
 \begin{align}
y \leq \lfloor \frac{x+z}{k} \rfloor.
 \end{align}
 Suppose that we move from $(x,y,z)$ to $(p,q,r)$, i.e., $(p,q,r) \in movek((x,y,z))$. We prove that $(p,q,r) \in B_k$.
 
 \noindent
 Since $movek((x,y,z))=\{(u,y,z);u<x\} \cup \{(x,v,z);v<y\} \cup \{(x,y,w);w<z\} \cup \{(u,\min(y, \lfloor \frac{u+z}{k} \rfloor ),z);u<x\} \cup \{(x,\min(y, \lfloor \frac{x+w}{k} \rfloor ),w);w<z\}$, where $u,v,w \in Z_{\ge 0}$,
 we have one of the following cases:
 
\noindent
$(1)$ $(p,q,r)= (u,y,z)$ with $u<x$;
 
\noindent
$(2)$ $(p,q,r)= (x,v,z)$ with $v<y$;
 
\noindent
$(3)$ $(p,q,r)= (x,y,w)$ with $w<z$;
 
\noindent
$(4)$ $(p,q,r)= (u,\min(y,\lfloor \frac{u+z}{k} \rfloor ),z)$ with $u<x$; or 

\noindent
$(5)$ $(p,q,r)= (x,\min(y,\lfloor \frac{x+w}{k} \rfloor ),w)$ with $w<z$.

\noindent
For each of these cases, we can use Theorem \ref{theoremfor4mplus3a} to get $p \oplus q \oplus r \oplus \ne 0$.
\end{proof}

Next, we prove that starting with an element of $B_k$, there is a proper move that leads to an element of $A_k$.

\begin{thm}\label{thforkfrBtoA}
Let $(x,y,z) \in B_k$, then $movek((x,y,z))\cap A_k \ne \phi$.
\end{thm}

\begin{proof}
Let $(x,y,z) \in B_k$. Then, we have
\begin{align}
x \oplus y \oplus z \oplus \ne 0
\end{align}
and
\begin{align}
y \leq \lfloor \frac{x+z}{k} \rfloor.
\end{align}
Then, we have one of the five cases of Theorem \ref{theoremfor4mplus3b}.
Since $movek((x,y,z))=\{(u,y,z);u<x\} \cup \{(x,v,z);v<y\} \cup \{(x,y,w);w<z\} \cup \{(u,\min(y, \lfloor \frac{u+z}{k} \rfloor ),z);u<x\} \cup \{(x,\min(y, \lfloor \frac{x+w}{k} \rfloor ),w);w<z\}$, there exists $(p,q,r) \in movek ((x,y,z))$ such that 
$p\oplus q\oplus r= 0$. Therefore, $(p,q,r) \in movek((x,y,z))\cap A_k$
\end{proof}

By Theorem \ref{thforkfrAtoB} and \ref{thforkfrBtoA}, we finish the proof of Theorem \ref{theoremforkyxzchocol}. Starting the game with a position $(x,y,z)\in A_{k}$, by Theorem \ref{thforkfrAtoB}, any option (move) leads to position $(p,q,r)$ in $B_k$. From this position $(p,q,r)$, by Theorem \ref{thforkfrBtoA}, the opposing player can choose a proper option that leads to a position in $A_k$. Note that any option reduces some of the numbers of the position. In this way, the opposing player can always reach a position in $A_k$, winning by reaching $(0,0,0)\in A_{k}$. Therefore, $A_k$ is the set of $\mathcal{P}$-positions.

Starting the game with a position $(x,y,z)\in B_{k}$, by Theorem \ref{thforkfrBtoA}, we can choose a proper option that leads to a position $(p,q,r)$ in $A_k$. From $(p,q,r)$, any option by the opposing player leads to a state in $B_k$. In this way, we win the game by reaching $(0, 0, 0)$. Therefore, $B_k$ is the set of $\mathcal{N}$-positions.

By Theorem \ref{theoremforkyxzchocol}, $(x,y,z)$ is a $\mathcal{P}$-position if and only if $x\oplus y \oplus z=0 $. Then,
it is natural to wonder whether the Grundy number of a position $(x,y,z)$ is equal to $x \oplus y \oplus z$.
The Grundy number of a position does not equal $x\oplus y\oplus z$, and Example \ref{exampofgrundynotnim} presents a counter example.

Example \ref{exambymathematica} shows that the number of positions whose Grundy number is equal to the nim-sum is smaller than the number of positions whose Grundy number is not equal to the nim-sum.

\begin{exam}\label{exampofgrundynotnim}
In the chocolate game that satisfies the inequality 
$y \leq \lfloor \frac{x+z}{3} \rfloor $, the Grundy number of a position $(x,y,z)$ is not always equal to $x \oplus y \oplus z$.
We show this by example. 
By Definition \ref{defofmexgrundy} and Definition \ref{defofmovek},
\begin{align}\label{grundydeff}
\mathcal{G}((x,y,z))= \textit{mex}\{\mathcal{G}((u,v,w)): (u,v,w) \in move((x,y,z))\}.
\end{align}

We calculate Grundy numbers for the positions of chocolates in Figures \ref{grundyex000}, \ref{grundyex100}, \ref{grundyex001}, \ref{grundyex101}, \ref{grundyex002}, \ref{grundyex102}, and \ref{grundyex112} by using (\ref{grundydeff}).

\noindent
$(i)$ Since $(0,0,0)$ is the end position, by Definition \ref{defofmexgrundy}, we have $\mathcal{G}((0,0,0)) = 0$. 

\noindent
$(ii)$ Here, $movek((1,0,0)) =\{(0,0,0)\}$ and $\mathcal{G}((0,0,0)) = 0$. Hence, by Definition \ref{defofmexgrundy}, $\mathcal{G}((1,0,0)) = 1$. 

\noindent
$(iii)$ Similarly, we have $\mathcal{G}((0,0,1)) = 1$. 

\noindent
$(iv)$ Here, $movek((1,0,1)) =\{(1,0,0),(0,0,1)\}$, $\mathcal{G}((1,0,0)) = 1$, and $\mathcal{G}((0,0,1)) = 1$. Hence, by Definition \ref{defofmexgrundy}, $\mathcal{G}((1,0,1)) = 0$.

\noindent
$(v)$ Here, $movek((0,0,2)) =\{(0,0,0),(0,0,1)\}$, $\mathcal{G}((0,0,0)) = 0$, and $\mathcal{G}((0,0,1)) = 1$. Hence, by Definition \ref{defofmexgrundy}. $\mathcal{G}((0,0,2)) = 2$.

\noindent
$(vi)$ Here, $movek((1,0,2)) =\{(1,0,0),(1,0,1),(0,0,2)\}$, $\mathcal{G}((1,0,0)) = 1$, $\mathcal{G}((1,0,1)) = 0$, and $\mathcal{G}((0,0,2)) = 2$. Hence, by Definition \ref{defofmexgrundy}, $\mathcal{G}((1,0,2)) = 3$.

\noindent
$(vii)$ Here, $movek((1,1,2)) =\{(1,0,0),(1,0,1),(1,0,2),(0,0,2)\}$, $\mathcal{G}((1,0,0)) = 1$, 
$\mathcal{G}((1,0,1)) = 0$, $\mathcal{G}((1,0,2)) = 3$, and $\mathcal{G}((0,0,2)) = 2$. Hence, by Definition \ref{defofmexgrundy}, $\mathcal{G}((1,1,2)) = 4$.

\noindent
By $(vii)$, we have $\mathcal{G}((1,1,2)) = 4 \ne 1 \oplus 1 \oplus 2$.
	
\begin{figure}[!htb]
	\begin{minipage}[!htb]{0.45\columnwidth}
		\centering
	\includegraphics[width=0.3\columnwidth,bb=0 0 57 34]{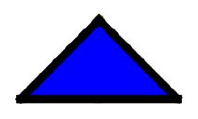}
	\caption{Position (0,0,0)}
	\label{grundyex000}
	\end{minipage}
	\begin{minipage}[!htb]{0.45\columnwidth}
		\centering
	\includegraphics[width=0.5\columnwidth,bb=0 0 56 31]{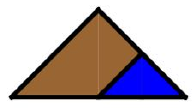}
	\caption{Position (1,0,0)}
	\label{grundyex100}
	\end{minipage}
\end{figure}

\begin{figure}[!htb]
	\begin{minipage}[!htb]{0.45\columnwidth}
		\centering
		\includegraphics[width=0.6\columnwidth,bb=0 0 56 31]{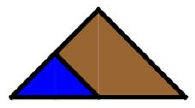}
		\caption{Position (0,0,1)}
		\label{grundyex001}
	\end{minipage}
	\begin{minipage}[!htb]{0.45\columnwidth}
		\centering
		\includegraphics[width=0.6\columnwidth,bb=0 0 85 46]{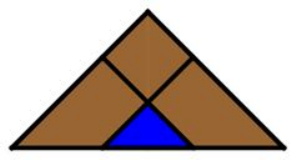}
		\caption{Position (1,0,1)}
		\label{grundyex101}
	\end{minipage}		
\end{figure}

\begin{figure}[!htb]
	\begin{minipage}[!htb]{0.45\columnwidth}
		\centering
		\includegraphics[width=0.6\columnwidth,bb=0 0 85 46]{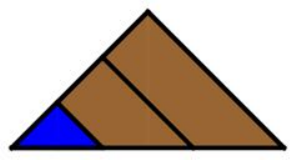}
		\caption{Position (0,0,2)}
		\label{grundyex002}
	\end{minipage}
	\begin{minipage}[!htb]{0.45\columnwidth}	
		\centering	
		\includegraphics[width=0.6\columnwidth,bb=0 0 57 24]{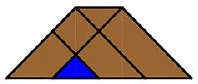}
		\caption{Position (1,0,2)}		
		\label{grundyex102}
	\end{minipage}
\end{figure}		

\begin{figure}[!htb]
		\centering
		\includegraphics[width=0.3\columnwidth,bb=0 0 57 30]{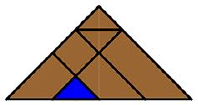}
		\caption{Position (1,1,2)}		
		\label{grundyex112}
\end{figure}		
	
In the next section, Example \ref{exambymathematica2} and Example \ref{examplecgsuitet2} present calculations by computer that show the Grundy number of a position $(x,y,z)$ is not always equal to $x \oplus y \oplus z$ in this game.
\end{exam}

\section{Computer Program for the Triangle Chocolate Bar Game}\label{sectionforcomputer}

\subsection{Computer Program to Calculate the $\mathcal{P}$-positions in the Chocolate Bar Game}
In this subsection, we present computer programs that show $\{(x,y,z);x,y,z\in Z_{\geq 0},y \leq \lfloor \frac{x+z}{k} \rfloor$ and $x\oplus y \oplus z=0\}$ is the set of $\mathcal{P}$-positions.

Example \ref{exambymathematica} presents a Mathematica program, and 
Example \ref{exambycgsuite} presents a Combinatorial Game Suite $($CGSuite$)$ program. 

\begin{exam}\label{exambymathematica}
	Here, let $k=3$.\\
	$(i)$ This Mathematica program
	presents the list $\{(x,y,z):x,y,z\in Z_{\geq 0},x \leq 20, y \leq 20, z \leq 20, y \leq \lfloor \frac{x+z}{k} \rfloor$, where $(x,y,z)$ is $\mathcal{P}$-position$\} - $
	$\{(x,y,z):x,y,z\in Z_{\geq 0},x \leq 20, y \leq 20, z \leq 20, y \leq \lfloor \frac{x+z}{k} \rfloor$ and $x\oplus y \oplus z=0\}$ that
	is a list of $\mathcal{P}$-positions whose nim-sum is not zero.
	
\begin{verbatim}
	k = 3; ss = 20; al = 
	Flatten[Table[{a, b, c}, {a, 0, ss}, {b, 0, ss}, {c, 0, ss}], 2];
	allcases = Select[al, (1/k) (#[[1]] + #[[3]]) >= #[[2]] &];
	move[z_] := Block[{p}, p = z;
	Union[Table[{t1, Min[Floor[(1/k) (t1 + p[[3]])], p[[2]]], 
	p[[3]]}, {t1, 0, p[[1]] - 1}],
	Table[{p[[1]], t2, p[[3]]}, {t2, 0, p[[2]] - 1}],
	Table[{p[[1]], Min[Floor[(1/k) (t3 + p[[1]])], p[[2]]], t3}, 
	{t3, 0, p[[3]] - 1}]]]
	Mex[L_] := Min[Complement[Range[0, Length[L]], L]];
	Gr[pos_] := Gr[pos] = Mex[Map[Gr, move[pos]]]
	pposition = Select[allcases, Gr[#] == 0 &];
	Select[pposition, BitXor[#[[1]], #[[2]], #[[3]]] > 0 &]
\end{verbatim}

The output shows that the list is empty, which implies that the nim-sum of a $\mathcal{P}$-position is zero.
	
\begin{verbatim}
{}
\end{verbatim}
	
	\noindent
	$(ii)$ The next Mathematica program presents the list $\{(x,y,z):x,y,z\in Z_{\geq 0},x \leq 20, y \leq 10, z \leq 20, y \leq \lfloor \frac{x+z}{k} \rfloor$, where $ (x,y,z)$ is $\mathcal{N}$-position$\} \cap \{(x,y,z):x,y,z\in Z_{\geq 0},x \leq 20, y \leq 20, z \leq 20, y \leq \lfloor \frac{x+z}{k} \rfloor$ and $x\oplus y \oplus z=0\}$
	that is a list of $\mathcal{N}$-positions whose nim-sum is zero.
	\begin{verbatim}
	Select[Complement[allcases, pposition], 
	BitXor[#[[1]], #[[2]], #[[3]]] == 0 &]
	\end{verbatim}
	This produces the following list.
	\begin{verbatim}
	{}
	\end{verbatim}
	The output shows that the list is empty, which implies that the nim-sum of a $\mathcal{N}$-position is not zero.
	
	By $(i)$ and $(ii)$, $(x,y,z)$ is a $\mathcal{P}$-position if and only if $x\oplus y \oplus z=0$.
\end{exam}

\begin{exam}\label{exambycgsuite}
Here, let $k=3$. This $($CGSuite version1.1.1$)$ program shows 
$\{(x,y,z);x,y,z\in Z_{\geq 0},y \leq \lfloor \frac{x+z}{k} \rfloor$ and $x\oplus y \oplus z=0\}$ is the set of $\mathcal{P}$-positions.

\noindent
$(i)$ First, we open the following file using CGSuite.
\begin{small}
\begin{verbatim}
class Choco3D extends ImpartialGame
var x,y,z,k;
method Choco3D(x,y,z,k)
end

override method Options(Player player)
result := [];

// x
for x1 from 0 to x-1 do
result.Add(Choco3D(x1,y.Min(((x1+z)/k).Floor),z,k));
end

// y
for y1 from 0 to y-1 do
result.Add(Choco3D(x,y1,z,k));
end

// z
for z1 from 0 to z-1 do
result.Add(Choco3D(x,y.Min(((x+z1)/k).Floor),z1,k));
end

result.Remove(this);

if x==0 and y==0 and z==0 then
return {};
else
return result;
end
end

override property ToString.get
return "Choco3D("+x.ToString+","+y.ToString+","
+z.ToString+","+k.ToString+")";
end
end
\end{verbatim}
\end{small}

\noindent
	$(ii)$ By typing the following command, we get the lists in $(a)$ and $(b)$.
	
\noindent
	$(a)$ $\{(x,y,z):x,y,z\in Z_{\geq 0},x \leq 20, y \leq 20, z \leq 20, y \leq \lfloor \frac{x+z}{k} \rfloor$, where $(x,y,z)$ is $\mathcal{P}$-position$\} - \{(x,y,z):x,y,z\in Z_{\geq 0}, x \leq 20, y \leq 20, z \leq 20, y \leq \lfloor \frac{x+z}{k} \rfloor$ and $x\oplus y \oplus z=0\}$ that
	is a list of $\mathcal{P}$-positions whose nim-sum is not zero. The output is an empty set. %setA
	
\noindent
	$(b)$ $\{(x,y,z):x,y,z\in Z_{\geq 0},x \leq 20, y \leq 20, z \leq 20, y \leq \lfloor \frac{x+z}{k} \rfloor$, where $ (x,y,z)$ is $\mathcal{N}$-position$\} \cap \{(x,y,z):x,y,z\in Z_{\geq 0}, x \leq 20, y \leq 20, z \leq 20, y \leq \lfloor \frac{x+z}{k} \rfloor$ and $x\oplus y \oplus z=0\}$
	that is a list of $\mathcal{N}$-positions whose nim-sum is zero.
	The output is an empty set. % setB	
	
	\begin{small}
		\begin{verbatim}
		x:=20;
		z:=20;
		y:=20;
		k:=3;
		setA:={};
		setB:={};
		for z1 from 0 to z do
			for x1 from 0 to x do
				for y1 from 0 to y.Min(((z1+x1)/k).Floor) do
					if examples.Choco3D(x1,y1,z1,k).CanonicalForm== 0 then
						if *x1+*y1+*z1!=0 then
							setA.Add([x1,y1,z1]);
						end
					else
						if *x1+*y1+*z1==0 then
							setB.Add([x1,y1,z1]);
						end
					end
				end
			end
		end
		Worksheet.Print("The Set of (Grundy Number = 0 and Nim-Sum > 0) -> " 
		+ setA.ToString);
		Worksheet.Print("The Set of (Grundy Number > 0 and Nim-Sum = 0) -> " 
		+ setB.ToString);
		\end{verbatim}
	\end{small}
	
\begin{figure}[!htb]
	\centering
	\includegraphics[width=0.8\columnwidth,bb=0 0 380 42]{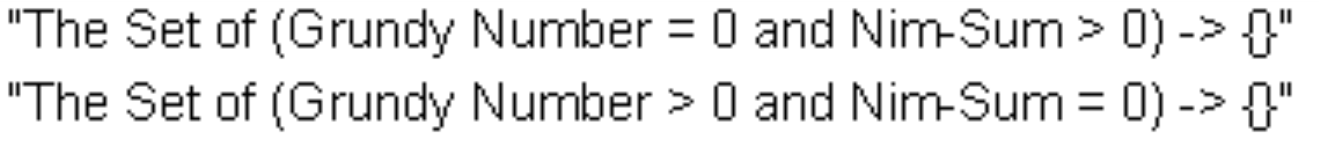}
	\caption{Result (1)}
	\label{CGSuiteResult1}
\end{figure}
	
Since the list in $(a)$ is empty, we have $x \oplus y \oplus z =0$ for any $\mathcal{P}$-position $(x,y,z)$.

Since the list in $(b)$ is empty, we have $x \oplus y \oplus z \ne 0$ for any $\mathcal{N}$-position $(x,y,z)$.

Therefore, $(x,y,z)$ is a $\mathcal{P}$-position if and only if $x \oplus y \oplus z = 0$.
\end{exam}

\subsection{Computer Program to Compare the Grundy Numbers and Nim-Sum of Positions in the Chocolate Bar Game}
In this subsection, we present computer programs that compare the Grundy number 
$\mathcal{G}((x,y,z))$ and $x\oplus y \oplus z$ for a position $(x,y,z)$.
Example \ref{exambymathematica2} presents a Mathematica program, and Example \ref{examplecgsuitet2} presents a CGSuite program.

\begin{exam}\label{exambymathematica2}
	Here, let $k=3$. This Mathematica program
	calculates the list $\mathcal{G}((x,y,z)):x,y,z\in Z_{\geq 0},x \leq 20, y \leq 20, z \leq 20, y \leq \lfloor \frac{x+z}{k} \rfloor\}$.
	\begin{verbatim}
	k=3;ss=20;al=
	Flatten[Table[{a,b,c},{a,0,ss},{b,0,ss},{c,0,ss}],2];
	allcases=Select[al,(1/k)(#[[1]]+#[[3]])>=#[[2]] &];
	move[z_]:=Block[{p},p=z;
	Union[Table[{t1,Min[Floor[(1/k)(t1+p[[3]])],p[[2]]],
	p[[3]]},{t1,0,p[[1]]-1}],
	Table[{p[[1]],t2,p[[3]]},{t2,0,p[[2]]-1}],
	Table[{p[[1]],Min[Floor[(1/k)(t3+p[[1]])],p[[2]]],t3},
	{t3,0,p[[3]]-1}]]]
	Mex[L_]:=Min[Complement[Range[0,Length[L]],L]];
	Gr[pos_]:=Gr[pos]=Mex[Map[Gr,move[pos]]]
	pposition=Select[allcases,Gr[#]==0 &];
	\end{verbatim}	
	
	\begin{verbatim}	
	nimequal=Select[allcases,BitXor[#[[1]],#[[2]],#[[3]]]==Gr[#] &]
	//Length
	\end{verbatim}
	The output of this code is the number of positions whose Grundy numbers are equal to nim-sum.
	\begin{verbatim}
	977
	\end{verbatim}	
	
	\begin{verbatim}
	nimnonequal=Select[allcases,!(BitXor[#[[1]],#[[2]],#[[3]]]
	==Gr[#]) &]//Length
	\end{verbatim}
	The output of this code is the number of positions whose Grundy numbers are not equal to nim-sum.
	\begin{verbatim}
	2257
	\end{verbatim}	
\end{exam}
	
\begin{exam}\label{examplecgsuitet2}
Here, let $k=3$. This CGSuite program
calculates the number of positions whose Grundy numbers are equal to nim-sum and the number of positions whose Grundy numbers are not equal to nim-sum. 

	\begin{small}
	\begin{verbatim}
		x:=20;
		z:=20;
		y:=20;
		k:=3;
		nimequal:=0;
		nimnonequal:=0;
		for z1 from 0 to z do
			for x1 from 0 to x do
				for y1 from 0 to y.Min(((z1+x1)/k).Floor) do
					if examples.Choco3D(x1,y1,z1,k).CanonicalForm==0 then
						if *x1+*y1+*z1==0 then
							nimequal:=nimequal+1;
						else
							nimnonequal:=nimnonequal+1;
						end
					else
						if examples.Choco3D(x1,y1,z1,k).CanonicalForm== *x1+*y1+*z1 then
							nimequal:=nimequal+1;
						else
							nimnonequal:=nimnonequal+1;		
						end
					end
				end
			end
		end
		Worksheet.Print("The Number of Grundy Number == Nim-Sum -> " 
		+ nimequal.ToString);
		Worksheet.Print("The Number of Grundy Number != Nim-Sum -> " 
		+ nimnonequal.ToString);
	\end{verbatim}
	\end{small}
	
\begin{figure}[!htb]
	\centering
	\includegraphics[width=0.8\columnwidth,bb=0 0 371 44]{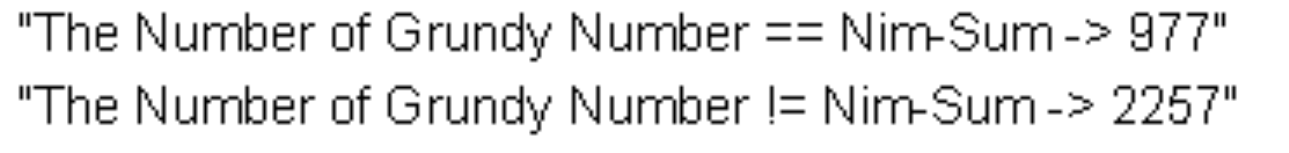}
	\caption{Result (2)}
	\label{CGSuiteResult2}
\end{figure}	
\end{exam}

\subsection{Computer Program to Calculate the $\mathcal{P}$-positions in the Chocolate Bar Game When {\boldmath $4m+1$} for Some {\boldmath $m \in Z_{\geq 0}$}}\label{4mplus1}
Let $k=4m+1$. In this subsection, we present computer programs that show 
$\{(x,y,z);x,y,z\in Z_{\geq 0}, y \leq \lfloor \frac{x+z}{k} \rfloor$ and $(x-1)\oplus y \oplus (z-1)=0\}$ is the set of $\mathcal{P}$-positions when $k=4m+1$ for some $m \in Z_{\geq 0}$.

Example \ref{exambymathematicafor5} presents a Mathematica program, and 
Example \ref{exambycgsuitefor5} presents a Combinatorial Game Suite (CGSuite) program.

\begin{exam}\label{exambymathematicafor5}
	Here, let $k=5$. 
	
	\noindent
	$(i)$ This Mathematica program
	presents the list $\{(x,y,z):x,y,z\in Z_{\geq 0},x \leq 20, y \leq 10, z \leq 20, y \leq \lfloor \frac{x+z}{k} \rfloor$, where $(x,y,z)$ is $\mathcal{P}$-position$\} - \{(x,y,z):x,y,z\in Z_{\geq 0},x \leq 20, y \leq 10, z \leq 20, y \leq \lfloor \frac{x+z}{k} \rfloor$ and $(x-1)\oplus y \oplus (z-1)=0\}$ that
	is a list of $\mathcal{P}$-positions $(x,y,z)$ such that $(x-1)\oplus y \oplus (z-1) \ne 0$.
	
	\begin{verbatim}
	k = 5; al = 
	Flatten[Table[{a,b,c},{a,0,20},{b,0,10},{c,0,20}],2];
	allcases=Select[al,(1/k)(#[[1]]+#[[3]])>=#[[2]] &];
	move[z_]:=Block[{p},p=z;
	Union[Table[{t1,Min[Floor[(1/k)(t1+p[[3]])],p[[2]]],
	p[[3]]},{t1,0,p[[1]]-1}],
	Table[{p[[1]],t2,p[[3]]},{t2,0,p[[2]]-1}],
	Table[{p[[1]],Min[Floor[(1/k)(t3+p[[1]])],p[[2]]],t3},
	{t3,0,p[[3]]-1}]]]
	Mex[L_]:=Min[Complement[Range[0,Length[L]],L]];
	Gr[pos_]:=Gr[pos]=Mex[Map[Gr,move[pos]]]
	pposition=Select[allcases,Gr[#]==0 &];
	Select[pposition,BitXor[#[[1]]-1,#[[2]],#[[3]]-1]>0 &]
	\end{verbatim}
	The output shows that the list is empty, which implies that the nim-sum $(x-1)\oplus y \oplus (z-1)=0$ for any $\mathcal{P}$-position $(x,y,z)$.
	\begin{verbatim}
	{}
	\end{verbatim}
	
	\noindent
	$(ii)$ The next Mathematica program presents the list $\{(x,y,z):x,y,z\in Z_{\geq 0},x \leq 20, y \leq 10, z \leq 20, y \leq \lfloor \frac{x+z}{k} \rfloor$, where $ (x,y,z)$ is $\mathcal{N}$-position$\} \cap $
	$\{(x,y,z):x,y,z\in Z_{\geq 0},x \leq 20, y \leq 10, z \leq 20, y \leq \lfloor \frac{x+z}{k} \rfloor$ and $(x-1)\oplus y \oplus (z-1)=0\}$
	that is a list of $\mathcal{N}$-positions $(x,y,z)$ such that $(x-1)\oplus y \oplus (z-1)=0$.
	\begin{verbatim}
	Select[Complement[allcases,pposition],
	BitXor[#[[1]]-1,#[[2]],#[[3]]-1]==0 &]
	\end{verbatim}
	This produces the following list.
	\begin{verbatim}
	{}
	\end{verbatim}
	
	The output shows that the list is empty, which implies that the nim-sum of a $\mathcal{N}$-position is not zero.\\
	By $(i)$ and $(ii)$, $(x,y,z)$ is a $\mathcal{P}$-position if and only if $(x-1)\oplus y \oplus (z-1)=0$.
\end{exam}

\begin{exam}\label{exambycgsuitefor5}
	Here, let $k=5$. \\
	$(i)$ This $($CGSuite version1.1.1$)$ program shows 
	$\{(x,y,z);x,y,z\in Z_{\geq 0},y \leq \lfloor \frac{x+z}{k} \rfloor$ and $(x-1)\oplus y \oplus (z-1)=0\}$ is the set of $\mathcal{P}$-positions.
	
	First, we open the code in $(i)$ of Example \ref{exambycgsuite}.
	
	\noindent
	$(ii)$
	By typing the following command, we get the list $\{(x,y,z):x,y,z\in Z_{\geq 0},x \leq 20, y \leq 10, z \leq 20, y \leq \lfloor \frac{x+z}{k} \rfloor, $ where $(x,y,z)$ is $\mathcal{P}$-position$\} - \{(x,y,z):x,y,z\in Z_{\geq 0},x \leq 20, y \leq 10, z \leq 20, y \leq \lfloor \frac{x+z}{k} \rfloor$ and $(x-1)\oplus y \oplus (z-1)=0\}$ that
	is a list of $\mathcal{P}$-positions $(x,y,z)$ such that $(x-1)\oplus y \oplus (z-1) \ne 0$. The output is an empty set.
	
	By typing the following command, we get the list $\{(x,y,z):x,y,z\in Z_{\geq 0},x \leq 20, y \leq 10, z \leq 20, y \leq \lfloor \frac{x+z}{k} \rfloor$, where $ (x,y,z)$ is $\mathcal{N}$-position$\} \cap \{(x,y,z):x,y,z\in Z_{\geq 0},x \leq 20, y \leq 10, z \leq 20, y \leq \lfloor \frac{x+z}{k} \rfloor$ and $(x-1)\oplus y \oplus (z-1)=0$
	that is a list of $\mathcal{N}$-positions $(x,y,z)$ such that $(x-1)\oplus y \oplus (z-1)=0$.
	The output is an empty set.
	
	\begin{small}
		\begin{verbatim}
		x:=20;
		z:=20;
		y:=10;
		k:=5;
		setA:={};
		setB:={};
		for z1 from 0 to z do
			for x1 from 0 to x do
				for y1 from 0 to y.Min(((z1+x1)/k).Floor) do
					if examples.Choco3D(x1,y1,z1,k).CanonicalForm==0 then
						if ((x1-1).NimSum(y1)).NimSum(z1-1)>0 then
							setA.Add([x1,y1,z1]);
						end
					else
						if ((x1-1).NimSum(y1)).NimSum(z1-1)==0 then
							setB.Add([x1,y1,z1]);
						end
					end
				end
			end
		end
		Worksheet.Print("The list of P-positions and Nim-Sum > 0 -> " 
		+ setA.ToString);
		Worksheet.Print("The list of N-positions and Nim-Sum = 0 -> " 
		+ setB.ToString);
		\end{verbatim}
	\end{small}

\begin{figure}[!htb]	
	\centering
	\includegraphics[width=0.8\columnwidth,bb=0 0 319 45]{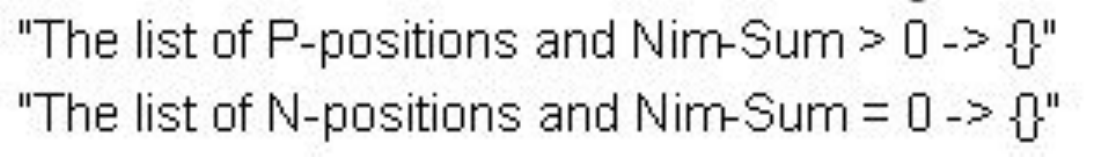}
	\caption{Result (3)}
	\label{CGSuiteResult1for5}
\end{figure}
\end{exam}

\begin{conjecture}
When $k = 4m+1$ for some $m \in Z_{ \geq 0}$,
$(x,y,z)$ is a $\mathcal{P}$-position if and only if $(x-1)\oplus y \oplus (z-1)=0$.
\end{conjecture}

\subsection{Computer Program to Calculate the $\mathcal{P}$-positions in the Chocolate Bar Game for Some Even Number {\boldmath $k$}}\label{evencase}

In this subsection, we present computer programs that present $\mathcal{P}$-positions of chocolate game when $k$ is an even number.

Example \ref{exambymathematicafor2} presents a Mathematica program, and 
Example \ref{exambycgsuitefor2} presents a Combinatorial Game Suite (CGSuite) program. 

\begin{exam}\label{exambymathematicafor2}
	Here, let $k=2$. This Mathematica program
	presents the list $\{(x,y,z):x,y,z\in Z_{\geq 0},x \leq 10, y \leq 10, z \leq 10, y \leq \lfloor \frac{x+z}{k} \rfloor$, where $(x,y,z)$ is a $\mathcal{P}$-position$\}$.
	
\begin{verbatim}
	k=2;ss=10;al=
	Flatten[Table[{a,b,c},{a,0,ss},{b,0,ss},{c,0,ss}],2];
	allcases=Select[al,(1/k)(#[[1]]+#[[3]])>=#[[2]] &];
	move[z_]:=Block[{p},p=z;
	Union[Table[{t1,Min[Floor[(1/k)(t1+p[[3]])],p[[2]]],
	p[[3]]},{t1,0,p[[1]]-1}],
	Table[{p[[1]],t2,p[[3]]},{t2,0,p[[2]]-1}],
	Table[{p[[1]],Min[Floor[(1/k)(t3+p[[1]])],p[[2]]],t3},
	{t3,0,p[[3]]-1}]]]
	Mex[L_]:=Min[Complement[Range[0,Length[L]],L]];
	Gr[pos_]:=Gr[pos]=Mex[Map[Gr,move[pos]]]
	pposition=Select[allcases,Gr[#]==0 &]
\end{verbatim}

	By the following output, there seems to be no generic formula for $ \mathcal{P}$-position.
	
\begin{verbatim}
{(0,0,0),(1,0,1),(1,1,2),(2,0,2),(2,1,1),(3,0,3),(3,1,4),
(3,2,5),(3,3,6),(3,4,7),(3,5,8),(4,0,4),(4,1,3),(4,2,6),
(4,3,5),(4,4,8),(4,5,7),(5,0,5),(5,1,6),(5,2,3),(5,3,4),
(5,4,9),(5,5,10),(6,0,6),(6,1,5),(6,2,4),(6,3,3),(6,4,10),
(6,5,9),(7,0,7),(7,1,8),(7,2,9),(7,3,10),(7,4,3),(7,5,4),
(8,0,8),(8,1,7),(8,2,10),(8,3,9),(8,4,4),(8,5,3),(9,0,9),
(9,1,10),(9,2,7),(9,3,8),(9,4,5),(9,5,6),(10,0,10),
(10,1,9),(10,2,8),(10,3,7),(10,4,6),(10,5,5)}
\end{verbatim}
\end{exam}

\begin{exam}\label{exambycgsuitefor2}
Here, let $k=2$. This $($CGSuite version1.1.1$)$ program presents the list $\{(x,y,z):x,y,z\in Z_{\geq 0},x \leq 10, y \leq 10, z \leq 10, y \leq \lfloor \frac{x+z}{k} \rfloor$, where $(x,y,z)$ is a $\mathcal{P}$-position$\}$.
First, we open the code in $(i)$ of Example \ref{exambycgsuite}.
By typing the following command, we get the list $\{(x,y,z):x,y,z\in Z_{\geq 0},x \leq 10, y \leq 10, z \leq 10, y \leq \lfloor \frac{x+z}{k} \rfloor$, where $(x,y,z)$ is a $\mathcal{P}$-position$\}$. 

By the following output, there seems to be no generic formula for $ \mathcal{P}$-position.

	\begin{small}
	\begin{verbatim}
	x:=10;
	z:=10;
	y:=10;
	k:=2;
	setA:={};
	for z1 from 0 to z do
		for x1 from 0 to x do
			for y1 from 0 to y.Min(((z1+x1)/k).Floor) do
				if examples.Choco3D(x1,y1,z1,k).CanonicalForm==0 then
					setA.Add([x1,y1,z1]);
				end
			end
		end
	end
	Worksheet.Print(setA.ToString);
	\end{verbatim}
\end{small}

\begin{figure}[!htb]
	\centering
	\includegraphics[width=0.9\columnwidth,bb=0 0 507 102]{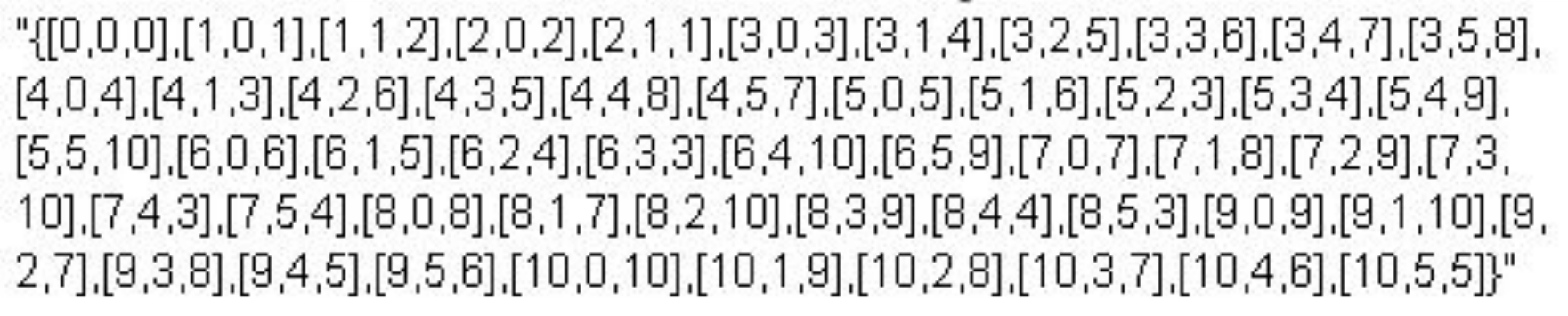}	
	\label{CGSuiteResult1for2}
	\caption{Result (4)}
\end{figure}
\end{exam}


\begin{thebibliography}{111}
	\bibitem{robin} A.C. Robin, A poisoned chocolate problem, Problem corner, \textit{The Mathematical Gazette}, \textbf{73}, (1989), 341-343. The answer for this problem can be found in \textbf{74}, (1990), 171-173.
	\bibitem{bouton} C. L. Bouton, Nim, a game with a complete mathematical theory, \textit{Annals of Mathematics}, \textbf{3(14)}, (1901), 35.
	\bibitem{integer2015} S. Nakamura, and R. Miyadera, Impartial Chocolate Bar Games, \textit{Integers} \textbf{15}, (2015), Article G4.
	\bibitem{integer2016} M. Inoue, M. Fukui and R. Miyadera, Impartial Chocolate Bar Games with a Pass, {\it Integers} \textbf{16}, (2016), Article G5.
	\bibitem{ipsj1} R. Miyadera, T. Inoue, W. Ogasa, and S. Nakamura, Chocolate Games that are Variants of Nim, Mathematics of Puzzles, \textit{Journal of Information Processing}, \textbf{53(6)}, (2012), 1582-1591 (in Japanese).
	\bibitem{gale} D. Gale, A curious Nim-type game, \textit{American Mathematical Monthly}, \textbf{81}, (1974), 876-879.
	\bibitem{lesson} M.H. Albert, R.J. Nowakowski, and D. Wolfe, \textit{Lessons In Play}, A K Peters/CRC Press, (2007).
	\bibitem{combysiegel} A.N. Siegel, \textit{Combinatorial Game Theory (Graduate Studies in Mathematics)}, American Mathematical Society, (2013).
\end{thebibliography}
\end{document}